\documentclass[final,1p,times]{elsarticle}
\usepackage{amscd} 
\usepackage{mathrsfs}
\usepackage[english]{babel}
\usepackage{amsfonts,amsmath,amsxtra,amsthm,amssymb,latexsym}
\usepackage{verbatim}
\biboptions{square,comma,numbers}

\newtheorem{theorem}{Theorem}[section]
\newtheorem{lemma}[theorem]{Lemma}
\newtheorem{corollary}[theorem]{Corollary} 
\theoremstyle{definition}

\newtheorem{remark}[theorem]{Remark} 
\numberwithin{equation}{section}

\def\z*{\bar z}

\def\B{\mathsf B}

\def\uno{\mathsf 1}

\def\C{\mathcal C}

\def\D{\mathscr D}
\def\dom{\text{\rm dom}}
\def\ran{\text{\rm ran}}
\def\supp{\text{\rm supp}}

\def\RE{\mathbb R}
\def\CO{{\mathbb C}}

\def\SL{S\! L}

\def\Sf{\mathbb S}

\def\ph*{\phi_\star}

\def\t {\tilde}
\def\be{\begin{equation}}
\def\ee{\end{equation}}

\def\-{{\rm in}}
\def\+{{\rm ex}}
\def\comp{{\rm comp}}
\def\loc{{loc}}

\journal{Journal of Differential Equations}
\begin{document}
\begin{frontmatter}

\author[man]{Andrea Mantile}
\ead{andrea.mantile@univ-reims.fr}
\address[man]{Laboratoire de Math\'{e}matiques, Universit\'{e} de Reims -
FR3399 CNRS, Moulin de la Housse BP 1039, 51687 Reims, France}

\author[pos]{Andrea Posilicano\corref{cor}}
\ead{andrea.posilicano@uninsubria.it}
\cortext[cor]{Corresponding author}
\address[pos]{DiSAT - Sezione di Matematica, Universit\`a dell'Insubria, Via Valleggio 11, I-22100 Como, Italy}

\author[sin]{Mourad Sini}
\ead{mourad.sini@oeaw.ac.at}
\address[sin]{RICAM, Austrian Academy of
Sciences, Altenbergerstr. 69, A-4040 Linz, Austria}
\title%[Uniqueness in inverse acoustic scattering]
{Uniqueness in inverse acoustic scattering with unbounded gradient across
Lipschitz surfaces}

\begin{abstract}
We prove uniqueness in inverse acoustic scattering in the case the density of
the medium has an unbounded gradient across $\Sigma\subseteq\Gamma
=\partial\Omega$, where $\Omega$ is a bounded open subset of $\mathbb{R}^{3}$
with a Lipschitz boundary. This follows from a uniqueness result in
inverse scattering for Schr\"odinger operators with singular
$\delta$-type potential supported on the surface $\Gamma$ and of strength
$\alpha\in L^{p}(\Gamma)$, $p>2$.

\end{abstract}

\begin{keyword}
Inverse Scattering \sep Acoustic equation \sep Schr\"odinger operators
%% or \MSC[2008] code \sep code (2010 is the default)
\MSC[2010] 81U40\sep 78A46\sep 35J10\sep 47B25

\end{keyword}

\end{frontmatter}
%\textbf{Preliminary draft - Not for distribution}
%\newpage

\begin{section}{Introduction. }
The aim of this paper is the study of the uniqueness problem in the inverse scattering for the acoustic wave equation
$$
\partial^{2}_{tt}\,u=v^{2}\varrho\,\nabla\!\cdot\!\left(\frac1{\varrho}\,{\nabla u}\right)
%\ \big(\equiv v^{2}\!\varrho\, L_{\varrho^{-\frac12}}p\,\big)
$$
in the case $\varrho$ has an unbounded gradient across some surface $\Sigma\subseteq\Gamma=\partial\Omega$, where $\Omega\subset\RE^{3}$ is open and bounded with Lipschitz boundary.
Here $u$ is the pressure field, $\varrho$ is the density and $v$ is the sound speed; we assume that $\varrho(x)=v(x)=1$ whenever $x$ lies outside some large ball $B_{R}\supset \Omega$. \par
To introduce our arguments and to allow the reasoning in the following lines, we start assuming that the functions $\varrho$ and $v$ are positive and sufficiently regular (for instance we can take $\varrho$ of class $C^2$ and $v$ bounded). Looking for fixed frequency solutions of the kind $u(t,x)= e^{-i\omega t}u_{\omega}(x)$, $\omega>0$, one gets the stationary equation
\be\label{acoustic}
-\,\omega^{2}u_{\omega}=v^{2}\varrho\,\nabla\!\cdot\!\left(\frac1{\varrho}\,{\nabla u_{\omega}}\right)\,.
\ee
Defining $\t u_{\omega}:={\varrho}^{-1} {u_{\omega}}$,
%$$\t u_{\omega}:=\frac1{\varrho}\ {u_{\omega}}\,,$$
the equation \eqref{acoustic} transforms into
\be\label{scr}
H_{\varphi,v,\omega}\t u_{\omega}=\omega^{2}\t u_{\omega}\,,
\ee
where $H_{\varphi,v,\omega}$ denotes the  Schr\"odinger operator
\begin{equation}\label{Schroedinger-model}
H_{\varphi,v,\omega}:=-\Delta+V_{\varphi,v,\omega}
\end{equation}
\begin{equation} \label{Schroedinger-potential}
V_{\varphi,v,\omega}:=V_{\!\varphi}+V_{v,\omega}\,,
\qquad V_{\varphi}:=\frac{\Delta\varphi}{\varphi}\,,\quad \varphi:=\frac1{\sqrt\varrho}\,,\quad V_{v,\omega}:=\omega^{2}\left(1-\frac1{v^{2}}\right)\,.
\end{equation}
Notice that, since $\varrho=v=1$ outside $B_{R}$, the potential $V_{\varphi,v,\omega}$ is compactly supported. \par
As well known from stationary scattering theory in quantum mechanics, whenever $V$ is a short-range potential, a generalized eigenfunction for the corresponding Schr\"odinger operator,
$(-\Delta+V)\psi_{k}=k^{2}\psi_{k}$, $k>0$, admits the outgoing representation
$$
\psi_{k}(x)=e^{ik\hat\xi\cdot x}+\frac1{(2\pi)^{3/2}}\,\frac{e^{ik|x|}}{|x|}\,s_{V}(k,\hat\xi,\hat x)+O\big(|x|^{-2}\big)\,,\quad \hat x:=\frac{x}{|x|}\,,
$$
where $s_{V}(k,\hat\xi,\hat\xi')$, $\hat\xi,\hat\xi'\in\Sf^{2}$, denotes the scattering amplitude (see e.g. \cite[page 425]{AJS}).
Since the solution $u_{\omega}$ of equation \eqref{acoustic} and the solution $u_{\omega}$ of the corresponding quantum scattering problem \eqref{Schroedinger-model}-\eqref{Schroedinger-potential} identify outside a ball, i.e.: $u_{\omega}(x)=\t u_{\omega}(x)$ for $|x|>R$, the above representation yields the asymptotic formula
$$
u_{\omega}(x)=e^{i\omega\hat\xi\cdot x}+\frac{e^{i\omega|x|}}{|x|}\ u^{\infty}_{\varrho,v}(\omega,\hat\xi,\hat x)+O\big(|x|^{-2}\big)\,,
$$
where the far-field pattern $u^{\infty}_{\varrho,v}$ is related to the scattering amplitude by the equality
\be\label{relation}
u^{\infty}_{\varrho,v}(\omega,\hat\xi,\hat \xi' )=\frac1{(2\pi)^{3/2}}\ {s_{V_{\varphi,v,\omega}}(\omega;\hat\xi,\hat\xi')}\,.
\ee
The inverse acoustic scattering problem consists in recovering the couple of functions $(\varrho,v)$ from the knowledge of the far-field pattern at some fixed frequencies; in particular, to recover the two independent functions $\varrho$ and $v$, one needs the knowledge of the far-field patters at least for two different frequencies $\omega$ and $\t\omega$. Clearly, the solvability of such an inverse problem requires a corresponding uniqueness result:
$$
\begin{cases} u^{\infty}_{\varrho_{1},v_{1}}(\omega,\cdot,\cdot)=u^{\infty}_{\varrho_{2},v_{2}}(\omega,\cdot,\cdot)\\ u^{\infty}_{\varrho_{1},v_{1}}(\t\omega,\cdot,\cdot)=u^{\infty}_{\varrho_{2},v_{2}}(\t\omega,\cdot,\cdot)&\end{cases} \quad\Longrightarrow\qquad \begin{cases}\varrho_{1}=\varrho_{2}&\\v_{1}=v_{2}\,.&\end{cases}
$$
By \eqref{relation}, this uniqueness issue is a consequence of an analogous result concerning Schr\"odinger operators:
\be\label{schr}
s_{V_{1}}(\omega,\cdot,\cdot)=s_{V_{2}}(\omega,\cdot,\cdot) \quad\Longrightarrow\quad V_{1}=V_{2}\,.
\ee
The justification of the uniqueness property (\ref{schr}) goes back to the pioneering works \cite{Novikov:1988,SU,Ramm:1988}. The idea is based on the orthogonality relation $\int_{B_R}(V_1(x)-V_2(x))u_1(x) u_2(x)\,dx=0$,
involving the total field solutions $u_j$ to the Schr\"odinger equation with potential $V_j$, and which can be derived from the equality of the far-field  patterns for the two frameworks. Then the strategy consists in constructing a specific set of solutions $u_j$, known as complex geometrical optics solutions or CGO's in short, and use them to deduce that $\widehat{V_1}=\widehat{V_2}$ (here $\widehat{\ }$ stands for the Fourier transform). Finally, the two equalities $V_{\varphi_{1},v_{1},\omega}=V_{\varphi_{2},v_{2},\omega}$ and $V_{\varphi_{1},v_{1},\t\omega}=V_{\varphi_{2},v_{2},\t\omega}$ entail $(\varrho_{1},v_{1})=(\varrho_{2},v_{2})$. \par
The aim of our work is to extend the above reasoning and conclusions to the case in which the density function $\varrho$ belongs to $H^{1}_{loc}(\RE^{3})$ and the jump of its normal derivative across some closed set $\Sigma\subset\Gamma$ belongs to $L^{p}(\Gamma)$, $p>2$, where $\Gamma$ is the Lipschitz boundary of some opened and bounded $\Omega\subset\RE^{3}$ (see Section 7 for the precise hypotheses and statements).
Under these conditions, the corresponding Schr\"odinger equation is modeled by a potential of the form (\ref{Schroedinger-model})-(\ref{Schroedinger-potential}) with and additive $\delta$-type potential supported on
$\Gamma$ with a strength belonging to $L^p(\Gamma),\; p>2$ (see Section 3). Hence, the inverse problem consists in extending the above approach (holding for regular perturbations) to the case of Schr\"odingers operators
with singular $\delta$-type potentials supported on $\Gamma$.  \par
%To that extend one needs to generalize the known results about such kind of operators to the case in which the strength of the $\delta$-type interaction is unbounded.
As the setting of the problem is motivated by many applications in sciences and engineering, after those mentioned works, a considerable effort was put to improve and refine these results to deal with potentials in more
general classes of functions and also other models as the electromagnetism and elasticity for instance. The reader can see the following references for more information \cite{CK, Isakov-book, Ramm-book,Uhlmann}. 
A model of particular interest is the EIT (Electrical Impedance Tomography) problem, also called Calder\'on's problem,  which consists in identifying the conductivity $\sigma$ using Cauchy data
$(u\arrowvert_{\partial \Omega}, \sigma \nabla u \cdot \nu\arrowvert_{\partial \Omega})$ of the solution of $\nabla\cdot \sigma \nabla u=0,$ in $\Omega\subset \mathbb{R}^3 $.
The uniqueness question of this problem is reduced, in the same way as described above, to the construction of the CGO's, see \cite{SU}, where $\sigma$ is a positive $C^2$-smooth function. 
The regularity of $\sigma$ is reduced to $C^{\frac{3}{2}+\epsilon}$ in \cite{Brown},
then to $C^{\frac{3}{2},\infty}$ in \cite{PPU} and to $C^{\frac{3}{2}, p},\; p \geq 6$ in \cite{Brown-Torres}. Finally in \cite{Haberman-Tataru,Caro-Rogers} this condition is reduced to $W^{1, \infty}$ and then to $W^{1, 3}$ in \cite{Hab} where
the CGO's are constructed allowing potentials of the form $\nabla \cdot f +h$, where $f\in L^{3}$ and $h\in L^{\frac{3}{2}}$ with compact supports. 
This last result is a key for us as $\delta$-type potentials, with strengths in $L^p(\Gamma)\; p>2$, can be cast in these forms (see Section 6). 
In particular, using CGO techniques, the analysis developed in Section 6 and Section 7 provides with a uniqueness result for the case of positive and bounded acoustic densities $\varrho$ which are in $ H^{1}_{loc}(\RE^{3})$ and
such that $|\nabla_{\Omega_{\-/\+}}\varrho\,|\in L^{4}(\Omega_{\-/\+})
\, $ and $ \Delta_{\Omega_{\-/\+}}\varrho\in L^{2}(\Omega_{\-/\+}) $, where $\Omega_{\-/\+}$ denotes the interior or the exterior of $\Omega$, while the normal derivatives across a closed subset $\Sigma$ of $\Gamma$ have jumps of regularity $L^{p}(\Sigma)$ with $p>8/3$ (see Theorem \ref{Theorem-Acoustic} and Remark \ref{SLA} for the details).  \par
Let us now discuss the forward problem and how we model the acoustic scattering with such regularity of the density.
There are several ways to study and describe the solutions of the forward acoustic scattering and generate the far-field  patterns. We mention the variation formulation, see \cite{Hahner, CC} for instance, which
reduces the problem to a bounded domain $\Omega$ by introducing a Dirichlet-Neumann map to the exterior problem, i.e. stated in $\mathbb{R}^3\setminus\overline{\Omega}$, where the  background is homogeneous.
A second approach consists in using integral equations; this allows to reduce the problem to inverting a Lippmann-Schwinger equation via the Fredholm alternative, see \cite{KL}. The approach requires, in addition to the regularity of the coefficients, a positivity of the contrast, i.e. in our case $v^2\rho=const.$ and $\rho<1$, see \cite{KL}. \par
In this paper we follow a different strategy and exploit the connection between the acoustic problem and the Schr\"odinger  one, providing the link between \eqref{acoustic} and \eqref{scr} in the case the density $\varrho$ is no more $C^{2}$ as supposed in the reasonings above. Due to the lack of regularity of $\varrho$, we use Schr\"oedinger operators with $\delta$-type potentials and unbounded strengths, thus generalizing previously known results about such kind of operators (see e.g. \cite{BLL}, \cite{JST} and references therein); for this class of operators we provide the rigorous construction as self-adjoint extensions of the symmetric operator $\Delta|\C^{\infty}_{comp}(\RE^{3}\backslash\Gamma)$.  The Schr\"odinger approach allows the use of techniques from quantum mechanical stationary scattering theory, in particular, by extending some results provided in \cite{JST}, we get a limiting absorption principle (LAP for short in the following) for our class of Schr\"odinger operators; as a consequence, the scattering amplitude is derived and used to define the acoustic far-field patterns. Let us remark that, by combining the results contained in \cite{RS1} with \cite[Theorem 16]{deift}, one could get a non-stationary scattering theory (i.e. the existence of the wave operators) directly for the acoustic model whenever $0<c_{1}\le \varrho, v\le c_{2}<+\infty$. Nevertheless, using the connection with Schr\"odinger operators, and the corresponding LAP, our approach has the advantage of easily providing with the acoustic far-field patterns in terms of the (quantum mechanical) scattering amplitude and results better suited for the study of the inverse scattering problem. \par

The paper is organized as follows. The self-adjoint realizations of such operators are provided in Section 2 and the existence of a limiting absorption principle for them
is given in Section 4. The proof of the  connection between Schr\"odinger operators with $\delta$-type potentials and acoustic operators with densities  with unbounded gradients is provided in Section 3.
In Section 5, we give sense to the far-field  through the construction of the generalized eigenfunctions. In section 6, we derive the uniqueness result for the Schr\"odinger model, as Theorem \ref{uniq},
and then we conclude the corresponding result for the acoustic model, as Theorem \ref{Theorem-Acoustic}, in Section 7.
\end{section}

\begin{section} {Schr\"odinger operators with delta interactions of unbounded strength.}
%In order to have a well-defined stationary scattering theory (i.e. LAP, generalized eigenfunctions,...) for the Schr\"odinger operator associated with the operator $L_{\varphi}$ one need precise informations on the operators themselves. {\ } We are interested in the case in which $\nabla\varphi$ has discontinuities across an surface. {\ }
%\vskip4pt
%Now we recall the definition of the self-adjoint operator describing a singular perturbation of a Schr\"odinger operator with a delta potential supported on a (smooth) compact surface   $\Gamma=\partial\Omega$. Our reference operator is $A_{V}:=\Delta-V$.{\ }\vskip4pt
Let $V\in L^{2}(\RE^{3})+L^{\infty}(\RE^{3})$; then, by the Kato-Rellich theorem,
$$A_{V}: H^{2}(\RE^{3})\subset L^{2}(\RE^{3})\to L^{2}(\RE^{3})\,,\quad A_{V}:=\Delta-V$$
is self-adjoint and bounded from above. Here $H^{s}(\RE^{3})$, $s\in \RE$, denotes the scale of Sobolev Spaces on $\RE^{3}$, we refer to \cite[Chapter 1]{Gri} for the appropriate definitions of such spaces and for the trace maps defined on them; we also refer to the same book for the definition of the scale $H^{s}(\Gamma)$, $|s|\le 1$, of Sobolev spaces on the Lipschitz surface $\Gamma$ which we use below.\par
$A_{V}$ can be broadened to an operator in $H^{-2}(\RE^{3})$ (by a slight abuse of notation we denote such an operator with the same symbol):
$$
A_{V}:L^{2}(\RE^{3})\subset H^{-2}(\RE^{3})\to H^{-2}(\RE^{3})\,,\quad A_{V}:=\Delta - V   ,
$$
where now $V$ denotes  the linear operator, belonging to $\B(L^{2}(\RE^{3}), H^{-2}(\RE^{3}))$ by the Kato-Rellich hypothesis, defined by
$$
\langle Vu,v\rangle_{H^{-2},H^2}:=\langle u,Vv\rangle_{L^{2}}\,,\quad u\in L^{2}(\RE^{3}),\ v\in H^{2}(\RE^{3})\subset L^{\infty}(\RE^{3})\,.
$$
Since $A_{V}\in B(H^{2}(\RE^{3}), L^{2}(\RE^{3}))$, by duality and interpolation  one has
$$A_{V}\in\B(H^{-s+2}(\RE^{3}),H^{-s}(\RE^{3}))\,,\quad 0\le s\le 2$$
and, setting $R_{z}^{V}:=(-A_{V}+z)^{-1}$, $z\in \rho(A_{V})$,
$$
\|R_{z}^{V}\|_{\B(H^{-s}(\RE^{3}),H^{-s+2}(\RE^{3}))}\le \|R_{z}^{V}\|_{\B(L^{2}(\RE^{3}),H^{2}(\RE^{3}))}\,,\quad 0\le s\le 2\,.
$$
\begin{lemma}\label{dzv} Let $d_z^{V}$ denote the distance of $z\in\rho(A_{V})$ from $\sigma(A_{V})$.
Then there exists $c_{V}>0$ such that, whenever $d_z^{V}>c_{V}$,
$$
\|R_{z}^{V}\|_{\B(H^{-s}(\RE^{3}),H^{-s+t}(\RE^{3}))}\le \frac1{(d_z^{V})^{1-\frac{t}2}}\,,\qquad 0\le s\le 2\,,\ 0\le t\le 2\,.
$$
\end{lemma}
\begin{proof} By $R^{V}_{z}=R^{0}_{z}(\uno+VR^{0}_{z})^{-1}$ and $\|VR^{0}_{z}\|_{\B(L^{2}(\RE^{3}))}\to 0$ as $|z|\to+\infty$, one gets, whenever $d_z^{V}>c_{V}$,
$$
\|R_{z}^{V}\|_{\B(L^{2}(\RE^{3}),H^{2}(\RE^{3}))}\le 1\,.
$$
Thus, since
$$
\|R_{z}^{V}\|_{\B(L^{2}(\RE^{3}))}\le \frac1{d_z^{V}}\,,
$$
by interpolation one obtains
$$
\|R_{z}^{V}\|_{\B(L^{2}(\RE^{3}),H^{t}(\RE^{3}))}\le \frac1{(d_z^{V})^{1-\frac{t}2}}
$$
and, by duality,
$$
\|R_{z}^{V}\|_{\B(H^{-s}(\RE^{3})),L^{2}(\RE^{3}),}\le \frac1{(d_{z}^{V})^{1-\frac{s}2}}\,.
$$
The proof is then concluded by interpolation again.
\end{proof}
Given $\Omega\subset \RE^{3}$, open and bounded with Lipschitz boundary $\Gamma$, we
introduce the bounded and surjective trace map
$$
\gamma_{0}:H^{s+\frac12}(\RE^{3})\to H^{s}(\Gamma)\,, \quad 0<s< 1\,,
$$
defined as the unique bounded extension of the map
$$
\gamma^{\circ}_{0}:C^{\infty}_{comp}(\RE^{3})\to C(\Gamma)\,,\quad \gamma_{0}^{\circ}u(x)=u(x)\,,\ x\in \Gamma\,.
$$
In the following we also use the extension (denoted by the same symbol) of $\gamma_{0}$ to  $H_{loc}^{s+\frac12}(\RE^{3})$ defined by $\gamma_{0}u:=\gamma_{0}(\chi u)$, where $\chi\in \C^{\infty}_{comp}(\RE^{3})$ and $\chi=1$ on an open neighborhood of $\Gamma$.\par
Using the adjoint $\gamma_{0}^{*}:H^{-s}(\Gamma)\to H^{-s-\frac12}(\RE^{3})$ and $R_{z}^{V}$ we define the bounded operator (the single-layer potential)
$$
\SL^{V}_{z}:=R^{V}_{z}\gamma_{0}^{*}: H^{-s}(\Gamma)\to H^{\frac32-s}(\RE^{3})
\,,\quad 0<s< 1\,.
$$
This gives the bounded operator
$$
\gamma_{0}\SL^{V}_{z}: H^{-s}(\Gamma)\to H^{1-s}(\Gamma)
\,,\quad 0<s<1\,.
$$
\begin{remark} Given $\phi\in H^{2}(\RE^{3})$ and $\xi\in H^{s}(\Gamma)$, $|s|\le 1$, let $\psi:=\phi-\SL^{V}_{z}\xi$. By the definition of $\SL_{z}^{V}$ one has $(-A_{V}+z)\psi=
(-A_{V}+z)\phi-\gamma_{0}^{*}\xi$. Thus, notwithstanding neither  $A_{V}\psi$ nor $\gamma_{0}^{*}\xi$ belong to $L^{2}(\RE^{3})$, one has  $$A_{V}\psi-\gamma_{0}^{*}\xi\in L^{2}(\RE^{3})\,.$$
\end{remark}
\begin{lemma}\label{2.3} Let $\alpha\in\B(H^{s}(\Gamma),H^{-s}(\Gamma))$, $0<s<\frac12$. Then there exists $c_{\alpha,V}>0$ such that forall $z\in\CO$ such that $d^{V}_{z}>c_{\alpha,V}$ one has $(\uno+\gamma_{0}\SL_{z}^{V}\alpha)^{-1}\in \B(H^{s}(\Gamma))$.
\end{lemma}
\begin{proof} By Lemma \ref{dzv}, one has
$$
\|R_{z}^{V}\|_{\B(H^{-s-1/2}(\RE^{3}),H^{s+1/2}(\RE^{3}))}\le \frac1{(d_z^{V})^{\frac12-s}}\,,\qquad 0\le s\le \frac12\,.
$$
Thus
$$
\|\gamma_{0}R_{z}^{V}\gamma_{0}^{*}\|_{\B(H^{-s}(\Gamma),H^{s}(\Gamma))}\le \frac1{(d_z^{V})^{\frac12-s}}\ \|\gamma_{0}\|^{2}_{\B(H^{s+1/2}(\RE^{3}),H^{s}(\Gamma))}\,,\qquad 0< s\le \frac12\,.
$$
Such an inequality show that if $0<s<\frac12$ then there exists $c_{\alpha,V}>0$ such that  operator norm $\|\gamma_{0}\SL_{z}^{V}\alpha\|_{\B(H^{s}(\Gamma))}$ is strictly smaller than one whenever $d^{V}_{z}>c_{\alpha,V}$.
\end{proof}
\begin{corollary}\label{coroll} Let $\alpha\in\B(H^{s}(\Gamma),H^{-s}(\Gamma))$, $0<s<\frac12$ such that $\alpha^{*}=\alpha$. Then there exists a finite set $S_{\alpha,V}\subset\RE$ such that $(\uno+\alpha\gamma_{0}\SL_{z}^{V})^{-1}\in \B(H^{-s}(\Gamma))$ for any $z\in \rho(A_{V})\backslash S_{\alpha,V}$. Moreover
\be\label{sym}
((\uno+\alpha\gamma_{0}\SL_{z}^{V})^{-1}\alpha)^{*}=(\uno+\alpha\gamma_{0}\SL_{\bar z}^{V})^{-1}\alpha\,.
\ee
\end{corollary}
\begin{proof} Let $0<s<\frac12$. By the compact embedding $H^{1-s}(\Gamma)\hookrightarrow H^{s}(\Gamma)$ and by $\ran(\gamma_{0}\SL_{z}^{V})\subseteq H^{1-s}(\Gamma)$, the map 
$\gamma_{0}\SL_{z}^{V}:H^{-s}(\Gamma)\to H^{s}(\Gamma)$ is compact and so  $\gamma_{0}\SL_{z}^{V}\alpha:H^{s}(\Gamma)\to H^{s}(\Gamma)$ is compact as well. 
Moreover, by the identity  $\SL_{z}^{V}=\SL_{w}^{V}+(w-z)R^{V}_{z}\SL_{w}^{V}$, the map $z\mapsto \gamma_{0}\SL_{z}^{V}\alpha$ is analytic from $\rho(A_{V})$ to $\B(H^{s}(\Gamma))$.
Thus, since the set of $z\in\rho(A_{V})$ such that $(\uno+\gamma_{0}\SL_{z}^{V}\alpha)^{-1}\in \B(H^{s}(\Gamma))$ is not void  by Lemma \ref{2.3}, by analytic Fredholm theory (see e.g. \cite[Theorem XIII.13]{RS}), $(\uno+\gamma_{0}\SL_{z}^{V}\alpha)^{-1}\in \B(H^{s}(\Gamma))$ for any $z\in \rho(A_{V})\backslash S_{\alpha,V}$, where $S_{\alpha,V}$ is a discrete set.  By next Theorem \ref{delta}, $S_{\alpha,V}$ is contained in the spectrum of a self-adjoint operator and so $S_{\alpha,V}\subset\RE$; hence, by Lemma \ref{2.3}, $S_{\alpha,V}\subseteq [\sup\sigma(A_{V}),\sup\sigma(A_{V})+c_{\alpha,V}]$ and so it is finite being discrete, i.e. without accumulation points.
\par
Since $(\uno+\gamma_{0}\SL_{\bar z}^{V}\alpha)^{*}=
(\uno+\gamma_{0}R_{\bar z}^{V}\gamma_{0}^{*}\alpha)^{*}=\uno+\alpha\gamma_{0}R_{z}^{V}\gamma_{0}^{*}=\uno+\alpha\SL_{z}^{V}$ and $\uno+\gamma_{0}\SL_{\bar z}^{V}\alpha$ is surjective, $\uno+\alpha\gamma_{0}\SL_{z}^{V}$ is injective and hence invertible for any $z\in \rho(A_{V})\backslash S_{\alpha,V}$. Moreover
$$
(\uno+\alpha\gamma_{0}\SL_{z}^{V})^{-1}=((\uno+\gamma_{0}\SL_{\bar z}^{V}\alpha)^{*})^{-1}
=\big((\uno+\gamma_{0}\SL_{\bar z}^{V}\alpha)^{-1}\big)^{*}\in\B(H^{-s}(\Gamma))\,.
$$
By the obvious equality $(\uno+\alpha\gamma_{0}\SL_{z}^{V})\alpha=\alpha(\uno+\gamma_{0}\SL_{z}^{V}\alpha)$, one gets $(\uno+\alpha\gamma_{0}\SL_{z}^{V})^{-1}\alpha=\alpha(\uno+\gamma_{0}\SL_{z}^{V}\alpha)^{-1}$ and so
\begin{align*}
((\uno+\alpha\gamma_{0}\SL_{z}^{V})^{-1}\alpha)^{*}=\alpha\big((\uno+\alpha\gamma_{0}\SL_{z}^{V})^{-1}\big)^{*}=\alpha(\uno+\gamma_{0}\SL_{\bar z}^{V}\alpha)^{-1}
=
(\uno+\alpha\gamma_{0}\SL_{\bar z}^{V})^{-1}\alpha\,.
\end{align*}
\end{proof}
By the previous results one has
\be\label{Z}
\CO\backslash\RE\subset  \rho(A_{V})\backslash S_{\alpha,V}\subseteq
Z_{V,\alpha}:=\{z\in\rho(A_{V}): (\uno+\alpha\gamma_{0}\SL_{z}^{V})^{-1}\in \B(H^{-s}(\Gamma))\}\,.
\ee
Thus
$$
Z_{V,\alpha}\not=\emptyset
$$
and
\be\label{krein}
R_{z}^{V,\alpha}:=R_{z}^{V}-\SL^{V}_{z}(\uno+\alpha\gamma_{0}\SL_{z}^{V})^{-1}\alpha\gamma_{0}R_{z}^{V}\,,\quad z\in Z_{V,\alpha}\,.
\ee
is a well-defined family of bounded operators in $L^{2}(\RE^{3})$.
\vskip8pt
Taking $\lambda_{\circ}\in \RE\cap\rho(A_{V})$, in the following we use the shorthand notation $\SL_{\circ}^{V}\equiv \SL_{\lambda_{\circ}}^{V}$.
\begin{theorem}\label{delta} Let $V\in L^{2}(\RE^{3})+L^{\infty}(\RE^{3})$ and  $\alpha\in\B(H^{s}(\Gamma),H^{-s}(\Gamma))$, $\alpha=\alpha^{*}$, $0<s<\frac12$. The family of bounded linear operators $R_{z}^{V,\alpha}$ given in \eqref{krein}
is the resolvent of the self-adjoint operator $A_{V,\alpha}$ in $L^{2}(\RE^{3})$ defined, in a $\lambda_{\circ}$-independent way, by
\be\label{dom}
\dom(  A_{V,\alpha})  :=\{ \psi\in H^{\frac32-s}(\RE^{3}): \psi+\SL_{{\circ}}^{V}\alpha\gamma_{0}\psi\in
H^{2}(  \mathbb{R}^{3})\}  \,,
\ee
\be\label{Aalfa}
A_{V,\alpha}\psi:=A_{V}\psi-\gamma_{0}^{*}\alpha\gamma_{0}\psi\,.
\ee
\end{theorem}
\begin{proof}
We proceed as in the proof of \cite[Theorem 2.1]{P2001}. Setting $\Lambda_{z}:=(\uno+\alpha\gamma_{0}\SL_{z}^{V})^{-1}\alpha$, using the resolvent identity for $R^{V}_{z}$ and definition \eqref{krein}, one gets, for any  $w,z\in Z_{V,\alpha}$ (see the explicit computation in
\cite[page 115]{P2001})%
\be\label{ide}
(  z-w)  R_{w}^{V,\alpha}R_{z}^{V,\alpha}=
R_{w}^{V,\alpha}  -R_{z}^{V,\alpha}-\SL_w^{V}\left((  \Lambda_{z}  -\Lambda_{w})
-(  z-w)  \Lambda_{w}  \gamma_{0}R_{{w}}^{V}\SL_{z}^{V}\Lambda_{z}\right)\gamma_{0}
R_{z}^{V}\,.
\ee
By $\SL_{z}^{V}=R_{z}^{V}\gamma_{0}^{*}$ and resolvent identity for $R^{V}_{z}$,  it results%
\begin{equation*}
(  1+\alpha\gamma_{0}\SL_{w}^{V})  -(  1+\alpha\gamma_{0}
\SL_{z}^{V})  =(  z-w)  \alpha\gamma_{0}R_w^{V}\SL_{z}^{V}\,.
\end{equation*}
This yields%
\begin{equation*}
\Lambda_{z}  -\Lambda_{w} =(  z-w)
\Lambda_{w}\gamma_{0}R_{w}^{V}\SL_{z}^{V}\Lambda_{z}
\end{equation*}
and \eqref{ide} reduces to
\begin{equation*}
(  z-w)  R_{w}^{V,\alpha}R_{z}^{V,\alpha}=R_{w}^{V,\alpha}-R_{z}^{V,\alpha}\,.
\label{Res_alpha_id}%
\end{equation*}
Therefore $R_{z}^{V,\alpha} $ is a pseudo-resolvent. Moreover,
$R_{z}^{V,\alpha}$ is injective, since, if $\psi\in\ker(R_{z}^{V,\alpha})$ then%
\begin{equation*}
R_z^{V} \psi=R_{z}^{V}\gamma_{0}^{*}\Lambda_{z}\gamma_{0}R_{z}^{V}\psi\,.
\end{equation*}
This gives $R_z^{V} \psi=0$ and so $\psi=0$. Hence, see e.g. \cite[Chap. VIII, Section 1.1]{Kato}, $R_{z}^{V,\alpha}  $ is the
resolvent of a closed operator $\hat A_{V,\alpha}$ and the identity \eqref{sym} implies%
\begin{equation*}
\big(R_{z}^{V,\alpha}\big)^*=R_{\bar{z}}^{V,\alpha}
\end{equation*}
so that such an operator is self-adjoint; given $z_{\circ}\in Z_{V,\alpha}$, $\hat A_{V,\alpha}$ is defined, in a $z_{\circ}$-independent way, by
\begin{equation}
\hat A_{V,\alpha}:=-(  R_{z_\circ}^{V,\alpha})  ^{-1}+z_{\circ}\,,\qquad\dom(  \hat A_{V,\alpha})  :=\ran(
R_{z_{\circ}}^{V,\alpha} )\,. \label{alpha_def}   %
\end{equation}
Notice that any $\psi\in\dom(  \hat A_{V,\alpha})  $ is given
by%
\begin{equation}
\psi=R_{z_{\circ}}^{V,\alpha}  \varphi=\psi_{z_{\circ}}-\SL_{z_{\circ}}^{V}\Lambda_{z_{\circ}}\gamma_{0}\psi_{z_{\circ}}\,,\quad\psi_{z_{\circ}}=R_z^{V}\varphi\in H^{2}(\RE^{3})\,,\ \varphi\in L^{2}(\RE^3)\,.
\end{equation}
By the mapping properties of $\SL_{z}$ and by $\ran(\Lambda_{z})\subseteq H^{-s}(\Gamma)$, one gets
$\dom(\hat A_{V,\alpha})=\ran(R^{V,\alpha}_{z_{\circ}})\subseteq H^{\frac32-s}(\RE^{3})$. Thus
$$
\dom(  \hat A_{V,\alpha})  :=\left\{ \psi\in H^{\frac32-s}(\RE^{3}): \psi=\psi_{z_{\circ}}-\SL_{z_{\circ}}^{V}(
1+\alpha\gamma_{0}\SL_{z_{\circ}}^{V})  ^{-1}\alpha\gamma_{0}\psi_{z_{\circ}}\,,\ \psi_{z_{\circ}}\in
H^{2}(  \mathbb{R}^{3})  \right\}  \,.
$$
%This decomposition is unique due to the injectivity of $R_{z}^{V,\alpha} $ and, for $\psi=R_{z_{\circ}}^{V,\alpha} \varphi$,
The definition (\ref{alpha_def}) yields%
\begin{equation}\label{op}
(  -\hat A_{V,\alpha}+z_{\circ})  \psi=(  -\hat A_{V,\alpha}+z_{\circ})  R_{z_{\circ}}^{V,\Gamma,\alpha
} \varphi=\varphi=(  -A_{V}+z_{\circ})  R_{z_{\circ}}^{V}\varphi=(  -A_{V}+z_{\circ})  \psi_{z_{\circ}}\,.
\end{equation}
Let us now show that $A_{V,\alpha}=\hat A_{V,\alpha}$.\par
Let $\psi=\psi_{z_{\circ}}-\SL_{z_{\circ}}^{V}\Lambda_{z_{\circ}}\gamma_{0}\psi_{z_{\circ}}\in \dom (\hat A_{V,\alpha})$. Since
\begin{equation}\label{trace}
{\alpha}\gamma_{0}\psi={\alpha}\gamma_{0}\psi_{z_{\circ}}-{\alpha}\gamma_{0}\SL_{z_{\circ}}^{V}(\uno+\alpha\gamma_{0}\SL_{z_{\circ}}^{V})  ^{-1}\alpha\gamma_{0}\psi_{z_{\circ}}
=(\uno+\alpha\gamma_{0}\SL_{z_{\circ}}^{V})  ^{-1}\alpha\gamma_{0}\psi_{z_{\circ}}=\Lambda_{z_{\circ}}\gamma_{0}\psi_{z_{\circ}}\,,
\end{equation}
one has $\psi=\psi_{z_{\circ}}-\SL_{z_{\circ}}^{V}\alpha\gamma_{0}\psi$. Then
\be\label{indep}
\psi+\SL_{{\circ}}^{V}\alpha\gamma_{0}\psi=\psi_{z_{\circ}}-(\SL_{z_{\circ}}^{V}-\SL_{\lambda_{\circ}}^{V})\alpha\gamma_{0}\psi=
\psi_{z_{\circ}}+(z_{\circ}-\lambda_{\circ})R^{V}_{\lambda_{\circ}}\SL_{z_{\circ}}^{V}\alpha\gamma_{0}\psi\in H^{2}(\RE^{3})
\ee
and so $\psi\in\dom(A_{V,\alpha})$. Conversely, given $\psi\in \dom(A_{V,\alpha})$, define $\psi_{z_{\circ}}:=\psi+\SL_{z_{\circ}}\gamma_{0}\psi$.
Then, by \eqref{trace}, $\psi=\psi_{z_\circ}+\SL_{z_{\circ}}\Lambda_{z_{\circ}}\gamma_{0}\psi_{z_{\circ}}$ and, by \eqref{indep}, $\psi_{z_{\circ}}\in H^{2}(\RE^{3})$. Thus
$\psi\in\dom(\hat A_{V,\alpha})$ and so $\dom(\hat A_{V,\alpha})=\dom(A_{V,\alpha})$.
By \eqref{op},
\begin{align*}
&\hat A_{V,\alpha}\psi=A_{V}\psi_{z_{\circ}}+z_{\circ}(\psi-\psi_{z_{\circ}})=A_{V}\psi_{z_{\circ}}+z_{\circ}\SL_{z_{\circ}}^{V}\Lambda_{z_{\circ}}\gamma_{0}\psi_{z_{\circ}}\\=&
A_{V}\psi_{z_{\circ}}+z_{\circ}\SL_{z_{\circ}}^{V}\alpha\gamma_{0}\psi
=A_{V}\psi+(-A_{V}+z_{\circ})\SL_{z_{\circ}}^{V}\alpha\gamma_{0}\psi
=A_{V}\psi+\gamma_{0}^{*}\alpha\gamma_{0}\psi\\=&A_{V,\alpha}\psi\,.
\end{align*}
Finally,
$$
\psi+\SL_{\mu_{\circ}}^{V}\alpha\gamma_{0}\psi=
\psi+\SL_{\lambda_{\circ}}^{V}\alpha\gamma_{0}\psi+(\lambda_{\circ}-\mu_{\circ})
R_{\lambda_{\circ}}^{V}\SL_{\mu_{\circ}}^{V}\alpha\gamma_{0}\psi
$$
shows that the definition of $\dom(A_{V,\alpha})$ is $\lambda_{\circ}$-independent.
\end{proof}
\begin{remark}\label{multi} A particular case of operator $\alpha\in \B((H^{s}(\Gamma),H^{-s}(\Gamma))$, such that $\alpha=\alpha^{*}$ is $\alpha\in M(H^{s}(\Gamma),H^{-s}(\Gamma)) $, $\alpha$ real-valued, where $M(H^{s}(\Gamma),H^{-s}(\Gamma))$ denotes the set of Sobolev multipliers on  $H^{s}(\Gamma)$ to $H^{-s}(\Gamma)$ (here and in the following we use the same notation for a function and for the corresponding multiplication operator).
By proceeding as in the proof of Theorem 2.5.3 in \cite{MS}, one has
$${|\alpha|^{1/2}}\in M(H^{s}(\Gamma),L^{2}(\Gamma))\quad\Longrightarrow\quad\alpha\in M(H^{s}(\Gamma),H^{-s}(\Gamma))\,.$$
Then, by Sobolev's embeddings and H\"older's inequality, one gets
$$p\ge \frac{1}{s} \quad\Longrightarrow\quad L^{p}(\Gamma)\subseteq M(H^{s}(\Gamma),H^{-s}(\Gamma)) \,.
$$
Thus we can define $A_{V,\alpha}$ whenever $\alpha\in L^{p}(\Gamma)$, $p>2$.
\end{remark}
\begin{remark}
One can check that, in the particular cases where  $V\in L^{\infty}(\RE^{3})$, $\alpha\in L^{\infty}(\Gamma)$ and $\Gamma$ is smooth, the self-adjoint operators $A_{V,\alpha}$ coincide with the ones studied (and constructed by different methods) in \cite[Section 3.2]{BLL}; also see \cite[Section 5.4]{JDE} for a construction that follows the lines here employed in the case $V\in C^{\infty}_{b}(\RE^{3})$, $\alpha\in M(H^{\frac32}(\Gamma))$ and $\Gamma$ is of class $C^{1,1}$. Similar kind of operators in the case $\Gamma$ is not necessarily Lipschitz and can have a not integer dimension have been considered in \cite[Example 3.6]{P2001}
\end{remark}
\begin{remark} Here and below we use dualities $\langle\cdot,\cdot\rangle_{X^{\!*}\!,X}$ which are conjugate linear with respect to the first variable. Let $\xi\in H^{-s}(\Gamma)$, $0< s\le 1$. Since
$$
\langle\gamma^{*}_{0}\xi,\phi\rangle_{H^{-s-1/2}(\RE^{3}),H^{s+1/2}(\RE^{3})}=\langle\xi,\gamma_{0}\phi\rangle_{H^{-s}(\Gamma),H^{s}(\Gamma)}
$$
for all $\phi\in H^{s+1/2}(\RE^{3})$, the distribution $\gamma^{*}_{0}\xi$ has support contained in $\Gamma$.  In the case $\xi\in L^{2}(\Gamma)$ one has
$$\langle\gamma^{*}_{0}\xi,\phi\rangle_{H^{-s-1/2}(\RE^{3}),H^{s+1/2}(\RE^{3})}=\int_{\Gamma}\bar\xi(x)\phi(x)\,d\sigma_{\Gamma}(x)\,,$$
where $\sigma_{\Gamma}$ denotes the surface measure. In particular $\gamma_{0}^{*}1=\delta_{\Gamma}$, where $\delta_{\Gamma}$ denotes the Dirac distribution supported on $\Gamma$.  Introducing the notation $\gamma_{0}^{*}\xi\equiv \xi\delta_{\Gamma}$, the operator $A_{V,\alpha}$ is represented as $A_{V,\alpha}\psi=A_{V}\psi-\alpha\gamma_{0}\psi\delta_{\Gamma}$ and this explain why this kind of operators are said to describe quantum mechanical models with singular, $\delta$-type interactions.
\end{remark}
\begin{remark} Notice that  $A_{V,\alpha}$ is a self-adjoint extension of the symmetric closed operator $A_{V}|\ker(\gamma_{0})$. If $\alpha\in M(H^{s}(\Gamma),H^{-s}(\Gamma))$ then $\supp(\gamma_{0}^{*}\alpha\gamma_{0}\psi)\subseteq\Sigma_{\alpha}$, $\Sigma_{\alpha}:=\supp(\alpha)$, and so $(A_{V,\alpha}\psi)|\Sigma_{\alpha}^{c}$ $=(A_{V}\psi)|\Sigma_{\alpha}^{c}$. This shows that $A_{V,\alpha}$ is a self-adjoint extension of the symmetric operator $A_{V}|\C^{\infty}_{comp}(\RE^{3}\backslash\Sigma_{\alpha})$ and so it depends only on $\Sigma_{\alpha}$ and not on the whole $\Gamma$: outside $\Sigma_{\alpha}$ we can change $\Gamma$ at our convenience without modifying the definition of $A_{V,\alpha}$.
\end{remark}
\begin{lemma}\label{spectrum} Under the assumptions of Theorem \ref{delta},
the self-adjoint operator $A_{V,\alpha}$ is bounded from above and $\sigma
_{ess}(  A_{V,\alpha})  =(  -\infty,0] $. Moreover, if $V$ is compactly supported and $\RE^{3}\backslash\overline\Omega$ is connected then $\sigma_{p}(  A_{V,\alpha})  \cap(  -\infty,0)
=\emptyset$.
\end{lemma}
\begin{proof} By $V\in L^{2}(\RE^{3})+L^{\infty}(\RE^{3})$ and by the Kato-Rellich theorem, $A_{V}$ is bounded from above. Thus, by \eqref{Z}, there exists $\lambda_{V}>\sup(
\sigma(  A_{V}))  $ such that: $\lambda\in Z_{V,\alpha}$ for
all $\lambda>\lambda_{V}$. Then, the resolvent formula (\ref{krein})
implies $(  \lambda_{V},+\infty)  \subset\rho(A_{V,\alpha}) $ and so $A_{V,\alpha}$ is bounded from above.\par
By Corollary \ref{coroll}, by the compact embedding $H^{-s}(\Gamma)\hookrightarrow H^{-1}(\Gamma)$ and by \eqref{krein}, the resolvent difference $R_z^{V,\alpha}  -R^{V}_{z}$ is a compact operator. Therefore, since $\sigma_{ess}(
A_{V}) =(  -\infty,0]$ (see e.g. \cite[Example 6, Section 4, Chapter XIII]{RS}), one has $\sigma_{ess}(  A_{V,\alpha})  =\sigma_{ess}(
A_{V})=(  -\infty,0]$.
\par Let us now suppose that $\supp(V)$ is compact and that exists $\lambda\in\sigma_{p}(  A_{V,\alpha})
\cap(  -\infty,0)  $; let $\psi_{\lambda}$ denote a corresponding eigenvector. Let $K$ a compact set containing both $\Gamma$ and $\supp(V)$, so that
$(-\Delta\psi_{\lambda}+\lambda\psi_{\lambda})|K^{c}=0$; by elliptic regularity, $\psi_{\lambda}\in C^{\infty}(K^{c})$,  and, by the Rellich estimate
one gets $\psi_{\lambda}|K^{c}=0$ (see e.g. \cite[Corollary 4.8]{Leis}).
Using the unique continuation property (holding for our exterior problem in $\RE^{3}\backslash\overline\Omega$ according to \cite{JerKo}), we get $\psi_{\lambda}|\RE^{3}\backslash\overline\Omega=0$. Since $\psi_{\lambda}\in \dom(A_{V,\alpha})\subseteq H^{\frac32-s}$, this gives $\gamma_{0}\psi_{\lambda}=0$ and so $\psi_{\lambda}\in H^{2}(\RE^{3})$ and $(-A_{V}+\lambda)\psi_{\lambda}=0$, i.e. $\lambda\in \sigma_{p}(A_{V})$. This contradicts $\sigma_{p}(  A_{V})  \cap(  -\infty,0)
=\emptyset$ (which holds for any $V\in L^{3/2}_{comp}(\RE^{3})$, see  \cite{JerKo}).
\end{proof}
The next lemma shows that the construction leading to Theorem \ref{delta} is unaffected by the addition of a bounded potential:
\begin{lemma}\label{bounded} Let $V$ and $\alpha$ be as in Theorem \ref{delta}. If $V_{\infty}\in L^{\infty}(\RE^{3})$ then
$$
A_{V,\alpha}+{V_{\infty}}=A_{V+V_{\infty},\alpha}\,.
$$
\end{lemma}
\begin{proof} According to the representation\eqref{Aalfa}, we only need to show that $\dom(A_{V+V_{\infty},\alpha})=\dom(A_{V,\alpha})$. By the definition of $\SL_{z}^{V}$ and the second resolvent identity there follows
$$
\SL_{z}^{V}-\SL_{z}^{V+V_{\infty}}=R_z^{V}V_{\infty}\SL_{z}^{V+V_{\infty}}\,.
$$
Since $R_z^{V}V_{\infty}\in\B(L^{2}(\RE^{3}),H^{2}(\RE^{3}))$, then \eqref{dom} yields the sought domains equality.
\end{proof}
\end{section}
\begin{section}{The connection between acoustic and Schr\"odinger operators.}
We begin the section by reviewing some results about multiplication of distributions and related topics.\par
%In the following we use the notation $\D(\mathcal O)\equiv C^{\infty}_{comp}(\mathcal O)$, $\mathcal O\subseteq\RE^{3}$ open, and $\D'(\mathcal O)$ denotes the space of distribution on $\mathcal O$.
Given the couple $u\in H_{\loc}^{t}(\RE^{3})$, $v\in H^{-s}(\RE^{3})$,  $0\le s\le t$, we can define the product $uv\in \D'(\RE^{3})$ by
$$
\langle uv,\phi\rangle_{\D',\D}:=\langle v,\phi \bar u\rangle_{H^{-s},H^{s}}\qquad \phi\in \D(\RE^{3})\,.
$$
In particular, the product $u(\gamma_{0}^{*}\xi)\in \D'(\RE^{3})$ is well defined for any
$\xi\in H^{-s}(\Gamma)$,   $0<s\le 1$, and $u\in H_{loc}^{t}(\RE^{3})$, $t\ge s+\frac12$.
\begin{lemma}
If $u\in H_{\loc}^{t}(\RE^{3})$, $v\in H^{-s}(\RE^{3})$,  $1\le s+1\le t$, then
\be\label{leibniz}
\nabla(uv)=(\nabla u)v+u\nabla v  \,.
\ee
\end{lemma}
\begin{proof}
\begin{align*}
\langle \nabla(uv),\psi\rangle_{\D',\D}=&-\langle uv,\nabla\phi\rangle_{\D',\D}=
-\langle v,\bar u\nabla\phi\rangle_{H^{-s},H^{s}}=-\langle v,\nabla(\bar u\phi )-\phi\nabla \bar u\rangle_{H^{-s},H^{s}}\\
=&\langle \nabla v,\phi \bar u\rangle_{H^{-s-1},H^{s+1}}+\langle(\nabla u)v,\phi\rangle_{\D',\D}=
\langle u\nabla v+(\nabla u)v,\phi\rangle_{\D',\D}\,.
\end{align*}
\end{proof}
\begin{remark}\label{rem-leibniz} Notice, that, by the same proof, \eqref{leibniz} holds true also in the case $u\in H^{1}_{comp}(\RE^{3})$ and $v\in L^{2}_{loc}(\RE^{3})$.
\end{remark}
\begin{lemma}\label{trace-prod}
If $u,v\in H^{1}(\RE^{3})$ then $uv\in W^{1,1}(\RE^{3})$ and $\gamma_{0}(uv)=\gamma_{0}u\gamma_{0}v$ in $L^{1}(\Gamma)$.
\end{lemma}
\begin{proof} By \eqref{leibniz}  , $uv\in W^{1,1}(\RE^{3})$. Since $\gamma_{0}\in \B(W^{1,1}(\RE^{3}), L^{1}(\Gamma))$ one has $\gamma_{0}(uv)\in L^{1}(\Gamma)$. Let $\{u_{n}\}_{1}^{\infty}\subset \D(\RE^{3})$, $\{v_{n}\}_{1}^{\infty}\subset \D(\RE^{3})$ such that $u_{n}\to u$ and $v_{n}\to v$ in $H^{1}(\RE^{3})$. Thus, by \eqref{leibniz} , $u_{n}v_{n}\to uv$ in $W^{1,1}(\RE^{3})$. Since  $\gamma_{0}\in \B(H^{1}(\RE^{3}), H^{\frac12}(\Gamma))$, $\gamma_{0}(u_{n}v_{n})=\gamma_{0}u_{n}\gamma_{0}v_{n}$ converges in $L^{1}(\Gamma)$ to both $\gamma_{0}(uv)$ and $\gamma_{0}u\gamma_{0}v$.
\end{proof}
Since $W^{1,\infty}(\Gamma)\subseteq M(H^{s}(\Gamma))$, $0\le s\le 1$, we can define the product $\zeta\xi\in W^{1,\infty}(\Gamma)'$ whenever $\zeta\in H^{t}(\Gamma)$ and $\xi\in H^{-s}(\Gamma)$, $0\le s\le t\le 1$,  by
$$
\langle \zeta\xi,f\rangle_{(W^{1,\infty})',W^{1,\infty}}:=\langle \xi,f \bar \zeta\rangle_{H^{-s},H^{s}}\qquad  f\in W^{1,\infty}(\Gamma)\,.
$$
Notice that the inclusion $W^{1,\infty}(\Gamma)\subset H^{1}(\Gamma)$ implies $H^{-s}(\Gamma)\subset W^{1,\infty}(\Gamma)'$, with $0\le s\le 1$.
Since $\gamma_{0}\phi\in W^{1,\infty}(\Gamma)$ whenever $\phi\in\D(\RE^{3})$, given $\xi\in W^{1,\infty}(\Gamma)'$ one defines $\gamma_{0}^{*}\xi\in \D'(\RE^{3})$ by
$$
\langle \gamma_0^{*}\xi,\phi\rangle_{\D',\D}:=\langle \xi,\gamma_{0}\phi\rangle_{(W^{1,\infty})',W^{1,\infty}}\,,\qquad \phi\in \D(\RE^{3})\,.
$$
In the case $\xi\in H^{-s}(\Gamma)$, $0<s\le 1$, the mapping properties of $\gamma_{0}$ imply
$\gamma_{0}^{*}\xi\in H^{-s-\frac12}(\RE^{3})$; then, from the above identity one recovers the preceding definition in term of the dual map of the trace $\gamma_{0}$.
\begin{lemma}\label{gamma} If $\xi\in H^{-s}(\Gamma)$, $0<s\le 1$, and $u\in H^{t+\frac12}(\RE^{3})$, $t\ge s$,  then
$$
u(\gamma_{0}^{*}\,\xi)=\gamma^{*}_{0}(\gamma_{0}u\,\xi)\,.
$$
\end{lemma}
\begin{proof}
\begin{align*}
\langle u(\gamma_{0}^{*}\xi),\phi\rangle_{\D',\D}=&\langle \gamma_{0}^{*}\xi,\phi \bar u\rangle_{H^{-s-1/2},H^{s+1/2}}=
\langle \xi,\gamma_{0}\phi \gamma_{0}\bar u)\rangle_{H^{-s},H^{s}}\\
=&
\langle \gamma_{0}u\,\xi,\gamma_{0}\phi\rangle_{(W^{1,\infty})',W^{1,\infty}}=
\langle \gamma_{0}^{*}(\gamma_{0}u\,\xi),\phi\rangle_{\D,\D'}\,.
\end{align*}
\end{proof}
\begin{lemma}\label{1/u}
Let $u\in H^{1}_{\loc}(\RE^{3})$ such that $\frac1u\in L^{\infty}(\RE^{3})$. Then $\frac1u\in H^{1}_{\loc}(\RE^{3})$ and
$$
\nabla \frac1u=-\frac{\nabla u}{u^{2}}\,,\qquad
\gamma_{0}\,\frac1u=\frac1{\gamma_{0}u}\,.
$$
\end{lemma}
\begin{proof}
Since $\frac1u\in L^{\infty}(\RE^{3})$, the definition of the distributional gradient
$$
\left\langle\nabla\frac1u,\phi\right\rangle_{\D',\D}=-\int_{\RE^{3}}\frac1{\bar u}\,\nabla\phi\, dx\,,\qquad \phi\in \D(\RE^{3})\,,
$$
shows that $\nabla\frac1u\in (W^{1,1}(\RE^{3}))'=W^{-1,\infty}(\RE^{3})$. Thus, for any $v\in W_{\loc}^{1,1}(\RE^{3})$, we can define the product $v\nabla\frac1u\in\D'(\RE^{3})$ by
$$
\left\langle v\nabla\frac1u,\phi\right\rangle_{\D',\D}=\left\langle\nabla\frac1u,\phi \bar v\right\rangle_{W^{-1,\infty},W^{1,1}}\,,\qquad \phi\in \D(\RE^{3})\,.
$$
Since $u\in H^{1}_{\loc}(\RE^{3})\subset W^{1,1}_{\loc}(\RE^{3})$, by
$$
0=\int_{\RE^{3}}\frac1{\bar u}\,(\bar u\nabla\phi)\,dx=
\int_{\RE^{3}}\frac1{\bar u}\,(\nabla(\bar u\phi)-\phi\nabla\bar u)\,dx\,,\qquad \phi\in \D(\RE^{3})\,,
$$
we get
$$
\left\langle\nabla\frac1u,\phi\bar u\right\rangle_{W^{-1,\infty},W^{1,1}}=
\left\langle u\nabla\frac1u,\phi \right\rangle_{\D',\D}=-\int_{\RE^{3}}\frac{\nabla\bar u}{\bar u}\,\phi\,dx\,,\qquad \phi\in \D(\RE^{3})\,,
$$
i.e. $u\nabla\frac1u=-\frac{\nabla u}{u}$. Let $\chi\in \C^{\infty}_{comp}(\RE^{3})$ such that $\chi=1$ on an open neighborhood of $\Gamma$; by Lemma \ref{trace-prod}, $1=\gamma_{0}(\chi u \chi \frac1u)=\gamma_{0}(\chi u)\gamma_{0}(\chi \frac1u)=\gamma_{0}u\gamma_{0}\frac1u$. Thus $\gamma_{0}u$ is a.e. different from zero and $\gamma_{0}\frac1u=\frac1{\gamma_{0}u}$.
\end{proof}
Given the real-valued function $\varphi$  we suppose there exists an open and bounded set $\Omega_{\varphi}\equiv\Omega\subset\RE^{3}$ with Lipschitz boundary $\Gamma_{\!\varphi}\equiv\Gamma$ such that
\be\label{H1}
\varphi\in H^{1}_{loc}(\RE^{3})\,,\quad\frac1\varphi\in L^{\infty}(\RE^{3})\,,\quad V_{\varphi}:=\frac1\varphi\,\big(\Delta_{\Omega_{\-}}(\varphi|\Omega_{\-}) \oplus\Delta_{\Omega_{\+}}(\varphi|\Omega_{\+})\big)\in L^{2}(\RE^{3})\,,
\ee
where $\Omega_{\-}\equiv \Omega$, $\Omega_{\+}\equiv\RE^{3}\backslash\overline\Omega$. Let $n(x)$ denote the exterior unit normal at $x\in\Gamma$; the lateral operators defined in  $C_{comp}^{\infty}(\overline\Omega_{\-/\+})$ by
$$\gamma_{1}^{\-/\+}u_{\-/\+}(x)=n(x)\cdot\nabla u_{\-/\+}(x)$$
uniquely extend to bounded maps
$$\gamma^{\-/\+}_{1}:H^{2}(\Omega_{\-/\+})\to H^{\frac12}(\Gamma)\,.$$
Furthermore, by \cite[Lemma 4.3 and Theorem 4.4]{McLe}, these extend to
$$
\hat\gamma^{\-/\+}_{1}:H^{1}_{\Delta}(\Omega_{\-/\+})\to H^{-\frac12}(\Gamma)\,,
$$
$$
H^{1}_{\Delta}(\Omega_{\-/\+}):=\{u_{\-/\+}\in H^{1}(\Omega_{\-/\+}):\Delta_{\Omega_{\-}}u_{\-/\+}\in L^{2}(\Omega_{\-/\+})\}\,,
$$
as bounded operator with respect to the natural norm $$\|u_{\-/\+}\|_{H^{1}_{\Delta}(\Omega_{\-/\+})}^{2}:=\|u_{\-/\+}\|_{H^{1}(\Omega_{\-/\+})}^{2}+\|\Delta_{\Omega_{\-/\+}}u_{\-/\+}\|_{L^{2}(\Omega_{\-/\+})}^{2}\,.$$
Therefore the jump across $\Gamma$ given by $$[\hat\gamma_{1}]\varphi:=\hat\gamma_1^{\+}\chi\varphi-\hat\gamma_1^{\-}\chi\varphi\,,$$ where $\chi\in \C^{\infty}_{comp}(\RE^{3})$ is such that $\chi=1$ on an open neighborhood of $\Gamma$, is a well-defined distribution in $H^{-\frac12}(\Gamma)$. Moreover, by Lemma \ref{1/u}, $\frac1{\gamma_{0}u}\in H^{\frac12}(\Gamma)$ and its  product with $[\hat\gamma_{1}]\varphi$ is well-defined in $W^{1,1}(\Gamma)'$. As further assumption, beside \eqref{H1}, we suppose
\be\label{H2}
\alpha_{\varphi}:=\frac{[\hat\gamma_{1}]\varphi}{\gamma_{0}\varphi}\in M(H^{s}(\Gamma),H^{-s}(\Gamma))\,,\quad s\in (0,1/2)\,.
\ee
In particular, by Remark \ref{multi}, hypothesis \eqref{H2} holds true whenever
$$
%\frac{[\hat\gamma_{1}]\varphi}{\gamma_{0}\varphi}
\alpha_{\varphi}\in L^{p}(\Gamma) \quad\text{ for some $p>2$.}
$$
\begin{remark}\label{SL} A more explicit characterization of a class of function $\varphi$ satisfying hypotheses \eqref{H1} and \eqref{H2} is the following:
$$
\varphi(x)=\varphi_{\circ}+\SL\xi\,,\qquad \SL\xi(x):=\int_{\Gamma}\frac{\xi(y)\,d\sigma_{\Gamma}(y)}{4\pi\,|x-y|}
$$
where $\varphi_{\circ}\in H^{2}_{loc}(\RE^{3})$ and $\xi\in L^{p}(\Gamma)$, $p>2$. By the properties of the single layer potential $\SL$ (see \cite[Theorem 3.1]{FMM}), one has
$$\Delta_{\Omega_{\-/\+}}\SL\xi=0\,,\quad [\hat\gamma_{1}]\SL\xi=-\xi\,,\quad\chi\SL\xi\in H^{1}(\RE^{3})\cap W^{1+1/p-\epsilon,p}(\Omega_{\-/\+})
$$
for any $\epsilon>0$ and any $\chi\in C_{comp}^{\infty}(\RE^{3})$. Since
$W^{1+1/p-\epsilon,p}(\Omega_{\-/\+})\subset C(\overline\Omega_{\-/\+})$ whenever $p>2$ and $\epsilon$ is sufficiently small, one gets $\varphi\in C(\RE^{3})$ and so $\varphi^{-1}\in L^{\infty}(\RE^{3})$ entails $(\gamma_{0}\varphi)^{-1}\in L^{\infty}(\Gamma)$. Thus, since $[\hat\gamma_{1}]\varphi_{\circ}=0$, one has $\alpha_{\varphi}=-\xi/\gamma_{0}\varphi\in L^{p}(\Gamma)\subset M(H^{s}(\Gamma),H^{-s}(\Gamma))$.
\end{remark}
%If $\varphi\ge c>0$ is continuous, so that $\frac1{\gamma_{0}\varphi}\in L^{\infty}(\Gamma)$, then
%hypothesis \eqref{H2} holds true whenever $${[\hat\gamma_{1}]\varphi}\in L^{p}(\Gamma)\quad \text{ for some $p>2$.}$$
By hypotheses \eqref{H1}, \eqref{H2} and Theorem \ref{delta}, we can introduce the self-adjoint operator in $L^{2}(\RE^{3})$ defined by
$$
A_{\varphi}:=A_{V_{\varphi},\alpha_{\varphi}}\,.
$$
The next theorem gives the connection between $A_{\varphi}$ and the acoustic operator:
\begin{theorem}\label{teo-acous} Let $\varphi$ satisfy hypotheses \eqref{H1} and \eqref{H2},  let
$\psi\in\dom(A_{\varphi})$ and set $u:={\varphi^{-1}}\psi$. Then $$\nabla u\in L^{1}_{\loc}(\RE^{3};\CO^{3})\,,\qquad\nabla\!\cdot\!\left(\varphi^{2}\nabla u\right)\in L^{2}(\RE^{3})
$$ and
$$
\frac1\varphi\nabla\!\cdot\!\left(\varphi^{2}\nabla u\right)=A_{\varphi}\psi\,.
$$
\end{theorem}
\begin{proof} By the ''half'' Green's formula (see \cite[Theorem 4.4]{McLe}, one gets
$$
\Delta\varphi=\Delta_{\Omega_{\-}}(\varphi|\Omega_{\-}) \oplus\Delta_{\Omega_{\+}}(\varphi|\Omega_{\+})+\gamma_{0}^{*}[\hat\gamma_{1}]\varphi\,.
$$
Thus $\Delta\varphi\in H^{-1}(\RE^{3})$ and, by \eqref{H1} and \eqref{H2}, we get
%$M(H^{s}(\Gamma),H^{-s}(\Gamma))\subseteq H^{-s}(\Gamma)$ and by Lemma \ref{gamma},
\be\label{Vphi}
V_{\varphi}=\frac1\varphi\,(\Delta\varphi-\gamma_{0}^{*}\alpha_{\varphi}\gamma_{0}\varphi)
%=\varphi (V_{\varphi}+\gamma_{0}^{*}\alpha_{\varphi})
\,.
\ee
Since both $\Delta\varphi$ and $\Delta\psi$ belong to $H^{-1}(\RE^{3})$ (notice that $\psi\in \dom(A_{\varphi})\subseteq H^{1}(\RE^{3})$), the products $\psi\Delta \varphi$ and $\varphi\Delta \psi$ are well-defined in $\D'(\RE^{3})$ and from\eqref{leibniz}  there follows
\be\label{nabla}
\varphi\Delta \psi- \psi\Delta\varphi=\nabla\!\cdot\!(\varphi\nabla \psi-\psi\nabla\varphi)\,.
\ee
Moreover, \eqref{leibniz}  and Lemma \ref{1/u} yield
$$
\nabla \frac\psi\varphi\in L^{1}_{\loc}(\RE^{3};\CO^{n})
$$
and, by \eqref{nabla} and \eqref{leibniz}, we get
$$
\frac1\varphi\nabla\!\cdot\!\left(\varphi^{2}\nabla\frac{\psi}\varphi\right)=
\frac1\varphi\nabla\!\cdot\!\left(\varphi\nabla \psi-\psi\nabla\varphi\right)=
\frac1\varphi\left(\varphi\Delta \psi-\psi\Delta\varphi\right)\,.
$$
Then, by Lemma \ref{trace-prod}, by \eqref{Vphi} and by \eqref{Aalfa},
\begin{align*}
&\frac1\varphi\nabla\!\cdot\!\left(\varphi^{2}\nabla\frac{\psi}\varphi\right)=\frac1\varphi\left(\varphi\Delta \psi-\psi\Delta\varphi\right)\\
=&\frac1\varphi\left((\varphi\Delta \psi-\gamma_{0}^{*}\alpha_{\varphi}\gamma_{0}\varphi\gamma_{0}\psi)-(\psi\Delta\varphi-\gamma_{0}^{*}\alpha_{\varphi}\gamma_{0}\varphi\gamma_{0}\psi)\right)\\
=&\frac1\varphi\left(\varphi(\Delta \psi-\gamma_{0}^{*}\alpha_{\varphi}\gamma_{0}\psi)-\psi(\Delta\varphi-\gamma_{0}^{*}\alpha_{\varphi}\gamma_{0}\varphi)\right)\\
=&\Delta \psi-\gamma_{0}^{*}\alpha_{\varphi}\gamma_{0}\psi-V_\varphi\psi
=A_{\varphi}\psi\,.
\end{align*}
\end{proof}
\end{section}
\begin{section} {The limiting absorption principle.}
In this section the results provided in \cite[Section 4]{JST}, which in particular apply to $A_{0,\alpha}$ (whenever $\alpha\in M(H^{\frac32}(\Gamma))$), are extended to $A_{V,\alpha}$.\par
The weighted Sobolev spaces $H_{w}^{k}(\RE^3)  $ are
defined for $k=0,1,2$ and $w\in\mathbb{R}$ by
$$
H_{w}^{s}(\RE^3)=\{u\in \D'(\RE^{3}):\|u\| _{H_{w}^{k}}(\RE^3)<+\infty\}\,,
$$
\begin{equation*}
\|   \varphi\|  ^{2} _{H_{w}^{k}(\RE^3)
}=\sum_{j=0}^{k}\|   \left\langle x\right\rangle ^{w}\nabla
^{j}u\|  ^{2} _{L^{2}(\RE^3)  }\,,\label{Sobolev_weighted}%
\end{equation*}
where $\left\langle x\right\rangle$ is a shorthand notation for the function $x\mapsto \left(  1+\|   x\|
^{2}\right)  ^{1/2}$. In particular, we set $L_{{w}}^{2}(\RE^3)  \equiv H_{{w}
}^{0}(\RE^3)  $. Since%
\begin{equation*}
\left[  \left\langle x\right\rangle ^{{w}},\partial_{i}\right]
\sim\left\langle x\right\rangle ^{{w}-1}\,,\quad\text{as }x_{i}%
\rightarrow0\,, \label{commutator}%
\end{equation*}
the two conditions $\left\langle x\right\rangle ^{{w}}u\in L^{2}(\RE^3)   $ and $\left\langle x\right\rangle ^{{w}}\nabla u\in L^{2}(\RE^3)  $ are equivalent to
$\left\langle x\right\rangle ^{{w}}u\in H^{1}(\RE^3) $; hence%
\begin{equation*}
H_{{w}}^{1}(\RE^3)  =\left\{  u\in{\D}%
^{\prime}(\RE^3)  :\left\langle x\right\rangle ^{{w}
}u\in H^{1}(\RE^3)  \right\}  \,.
\end{equation*}
A similar argument applies to $H_{{w}}^{2}(\RE^3)  $%
\begin{equation*}
H_{{w}}^{2}(\RE^3)  =\left\{  u\in{\D}%
^{\prime}(\RE^3)  :\left\langle x\right\rangle ^{{w}
}u\in H^{2}(\RE^3)  \right\}  \,.
\end{equation*}
In particular, this provide the equivalent $H_{{w}}^{2}(\RE^3) $-norm%
\begin{equation*}
|u|_{{H}_{{w}}^{2}(\RE^3)
}^{2}:=\int_{\mathbb{R}^{3}}\left\langle x\right\rangle ^{{w}}|
\left(  -\Delta+1\right)  u(x)|   ^{2}dx\,.
\label{Sobolev_weighted_norm_eq}%
\end{equation*}
The above definitions
are generalized to the case of non-integer order $s\in\mathbb{R}$ by%
\[
H_{{w}}^{s}(\RE^3)  :=\left\{  u\in{\D}%
^{\prime}(\RE^3)  :\left\langle x\right\rangle ^{{w}
}u\in H^{s}(\RE^3)  \right\}  \,,
\]
while the corresponding dual spaces (w.r.t. the $L^{2}$-product) identify with%
\begin{equation}
H_{-{w}}^{-s}(\RE^3)  =\left\{  u\in
{\D}^{\prime}(\RE^3)  :\left\langle
x\right\rangle ^{-{w}}u\in H^{-s}(\RE^3)
\right\}  \,. \label{Sobolev_weighted_dual}%
\end{equation}
For the open subset $\Omega\subset\mathbb{R}^{3}$, the spaces $H_{{w}}%
^{s}\left(  \Omega\right)  $ and $H_{{w}}^{s}\left(  \mathbb{R}^{3}%
\backslash\bar{\Omega}\right)  $ are defined in a similar way. In particular,
since $\Omega$ is bounded, one has: $H_{{w}}^{s}(\Omega)=H^{s}(\Omega)$, the
equalities holding in the Banach space sense; thus%
\begin{equation}
L_{{w}}^{2}(\RE^3)  =L^{2}\left(  \Omega\right)
\oplus L_{{w}}^{2}\left(  \RE^3 \backslash\overline\Omega\right)  \,,
\end{equation}
and
\begin{equation}
H_{{w}}^{s}\left(  \RE^3 \backslash\Gamma\right)  :=H^{s}\left(  \Omega\right)  \oplus H_{{w}}%
^{s}\left(  \RE^3 \backslash\overline\Omega\right)  \,.
\end{equation}
%We introduce the family of weighted spaces $L^{2}_{w}(\RE^{3})$ and $H^{2}_{w}(\RE^{3})$, defined, for any $w\in\RE$, by
%$$
%L^{2}_{w}(\RE^{3}):=\{u\in L^{2}_{\loc} (\RE^{3}):\|u\|_{L^{2}_{w}(\RE^{3})}<+\infty\}\,,
%$$
%$$
%\|   u\|  ^{2} _{L^{2}_{w}(   \RE^{3})
%}:=\int_{\RE^{3}}|u(x)|^{2}\,\left(1+\|u(x)\|^{2}\right)^{w}dx
%$$
%and
%$$
%H^{2}_{w}(\RE^{3}):=\{u\in H^{2}_{\loc} (\RE^{3}):\|(-\Delta+1)u\|_{L^{2}_{w}(\RE^{3})}<+\infty\}\,,
%$$
%%$$
%%\|   u\|  ^{2} _{H^{2}_{w}(   \RE^3 )
%%}:=\|u\|^{2}_{L^{2}_{w}(   \RE^3 )}+\sum_{1\le i\le n}\|\partial_{x_{i}}u\|^{2}_{L^{2}_{w}(   \RE^3 )}
%%+\sum_{1\le i,j\le n}\| \partial^{2}_{x_{i} x_{j}}u\|^{2}_{L^{2}_{w}(   \RE^3 )}
%%\,.
%%$$
%The spaces $L^{2}_{w}(\Omega_{{\+}})$ and $H^{2}_{w}(\Omega_{{\+}})$, are defined in a similar way.  {\ }\par Since $\Omega$ is bounded, one has
%$$L^{2}_{w}(\RE^{3})=L^{2}(\Omega_{\-})\oplus L^{2}_{w}(\Omega_{\+})$$ and
%$$H^{2}_{w}(\RE^{3}\backslash\Gamma)=H^{2}(\Omega_{\-})\oplus H^{2}_{w}(\Omega_{\+})\,.
%$$
The trace operators are extended to $H^{s}_{w}(\RE^{3}\backslash\Gamma)$, $w<0$, by   $$\gamma_{0}^{\+} u_{\+}:=\gamma_{0}^{\+}(\chi u_{\+}),\quad
\gamma_{1}^{\+} u_{\+}:=\gamma_{1}^{\+}(\chi u_{\+}),
$$
%\quad \gamma u:=\gamma(\chi u)\,,\quad [\gamma] u:=[\gamma](\chi u)\,,$$
where $\chi\in \C^{\infty}_{\comp}(\Omega^{c})$, $\chi=1$ on a neighborhood of $\Gamma$.
\par From now on we suppose that $V\in L^{2}_{comp}(\RE^{3})$, so that $\sigma_{p}(A_{V})\cap(-\infty,0)=\emptyset$ (see e.g. \cite{JerKo}) and, since $V$ is a short range potential, a limiting absorption principle (LAP for short) holds for $A_{V}$ (see e.g. \cite[Theorem 4.2]{Agm}):
\begin{theorem}\label{LAPV} Let $V\in L^{2}_{comp}(\RE^{3})$. For any $k\in\RE\backslash\{0\}$ and for any $w>{\frac12}$, the  limits
\be\label{lim1}
R^{V,\pm}_{-k^{2}}:=\lim_{\epsilon\downarrow 0}\, (-A_{V}-(k^{2}\pm i\epsilon))^{-1}
\ee
exist in $\B(L^{2}_{w}(\RE^{3}),H^{2}_{-w}(
\mathbb{R}^{3}) )$. Moreover
\begin{equation}
R^{V,\pm}_{  -k^{2}}  =R^{0,\pm}_{  -k^{2}}
-R^{0,\pm}_{  -k^{2}} VR^{V,\pm}_{  -k^{2}}  \,,
\label{Rpm}%
\end{equation}
and%
\begin{equation*}
(-\Delta+V-k^{2})  R^{V,\pm}_{  -k^{2}} =\uno\,.
\end{equation*}
%Setting $\CO_{\pm}:=\{z\in\CO:\pm\text{\rm Im}(z)> 0\}$ and
%$$
%R^{\pm}_{z}:=\begin{cases} (-\Delta+z)^{-1}\,,&z\in\CO_{\pm}\\
%R^{\pm}_{\lambda}\,,&\lambda\in(-\infty,0)\,,
%\end{cases}
%$$
%the maps $z\mapsto R^{\pm}_{z}$ are continuous on $\CO_{\pm}\cup(-\infty,0)$ to $\B(L^{2}_{\alpha}(\RE^{n}),H^{2}_{-\alpha}(\RE^3 ) )$.
\end{theorem}
\begin{remark} By duality, the limits \eqref{lim1} also exist in $\B(H^{-2}_{w}(\RE^{3}),L^{2}_{-w}(\mathbb{R}^{3}) )$ and so, by interpolation,
$$
R^{V,\pm}_{  -k^{2}} \in \B(H^{-s}_{w}(\RE^{3}),H^{-s+2}_{-w}(\mathbb{R}^{3}) ) \,,\quad 0\le s\le 2\,.
$$
\end{remark}
In order to extend LAP to operators of the kind $A_{V,\alpha}$, we need some preparatory lemmata. In the following $B_{R}$ denotes a sufficiently large ball such that $\supp(V)\subset B_{R}$.
\begin{lemma} Let $V\in L^{2}_{comp}(\RE^{3})$. Then, for all $z\in\rho(A_{V})$ and for all $w\in \RE$,
\begin{equation}
R_z^{V}  \in \B(  L_{w}^{2}(\RE^3),H_{{w}}^{2}(\RE^3))  \,.
\label{R_V_weight_est}%
\end{equation}
\end{lemma}
\begin{proof}
From the resolvent identity $R_z^{V}  =R_z^{0}  \left(  1-V   R_z^{V}  \right)$, there follows%
\[
\|   R_z^V   u\|   _{H_{{w}}^{2}(\RE^3)}\leq\|   R_z^0    u\|
_{H_{{w}}^{2}(\RE^3)  }+\|   R_{0}\left(
z\right)  V   R_z^V   u\|   _{H_{{w}}^{2}(\RE^3)}\,.
\]
Thus, since the thesis hold true in the case $V=0$ (this a a consequence of \cite[Lemma 1, page 170]{RS}, see the proof of Theorem 4.2 in \cite{JST}), we get%
\begin{equation}
\|   R_z^V   u\|   _{H_{{w}}^{2}(\RE^3)}\leq c\left(  \|   u\|
_{L_{{w}}^{2}(\RE^3)  }+\|   V   R_z^V  u\|   _{L_{{w}}^{2}(\RE^3)  }\right)
\,. \label{R_V_est_1}%
\end{equation}
Then the continuous injection $H^{2}\left(  B_{R}\right)  \hookrightarrow L^{\infty}\left(  B_{R}\right)  $ yields%
\[
\|   V   R_z^V   u\|   _{L_{{w}}^{2}(\RE^3)}\leq c\|   V   \|
_{L^{2}(\RE^3)  }\|   R_z^V
u\|   _{L^{\infty}\left(  B_{R}\right)  }\leq c\|   V   \|   _{L^{2}(\RE^3)  }\|
R_z^V   u\|   _{H^{2}\left(  B_{R}\right)  }\,.
\]
For ${w}\geq0$, the embedding $L_{{w}}^{2}(\RE^3)\hookrightarrow L^{2}(\RE^3)  $ and the standard mapping
properties of $R_z^V   $ lead to%
\begin{align}
\|   V   R_z^V   u\|   _{L_{{w}}%
^{2}(\RE^3)  }\leq \|
V   \|   _{L^{2}(\RE^3)  }\|  R_z^V u\|   _{H^{2}(\RE^3)  }   \leq\|   V   \|   _{L^{2}(\RE^3)  }\|   u\|   _{L^{2}(\RE^3)}\leq\|   V   \|   _{L^{2}(\RE^3)  }\|   u\|   _{L_{{w}}^{2}(\RE^3)  } \label{R_V_est_weight}%
\end{align}
and so, in this case, the statement follows from (\ref{R_V_est_1}) and
(\ref{R_V_est_weight}). For ${w}<0$ we proceed as in the proof of \cite[Lemma 1, page 170]{RS} starting from the identity%
\[
R_z^V   \left\langle x\right\rangle ^{|   {w}
|   }\left\langle x\right\rangle ^{{w}}u=\left\langle x\right\rangle
^{|   {w}|   }R_z^V   \left\langle
x\right\rangle ^{{w}}u+\big[  R_z^V   ,\left\langle
x\right\rangle ^{|   {w}|   }\big]  \left\langle
x\right\rangle ^{{w}}u\,.
\]
An explicit computation leads to%
\be
R_z^V   \left\langle x\right\rangle ^{|
{w}|   }\left\langle x\right\rangle ^{{w}}u=\big( \left\langle x\right\rangle ^{|   {w}|
}R_z^V   +R_z^V    \Delta\left\langle
x\right\rangle ^{|   {w}|   }  R_z^V
+2R_z^V    \nabla\left\langle x\right\rangle
^{|   {w}|   } \cdot\nabla R_z^V
\big)  \left\langle x\right\rangle ^{{w}}u
\ee
and so
\begin{align*}
\|   V   R_z^V   u\|   _{L_{{w}}^{2}(\RE^3)}  \leq&\|   V   \left\langle x\right\rangle
^{|   {w}|   }R_z^V    \left\langle
x\right\rangle ^{{w}}u  \|   _{L_{{w}}^{2}(\RE^3)}+ \|   V   R_z^V    \Delta\left\langle
x\right\rangle ^{|   {w}|   }  R_z^V
\left\langle x\right\rangle ^{{w}}u  \|   _{L_{{w}
}^{2}(\RE^3)  }\\
&+2\|   V   R_z^V   \nabla\left\langle x\right\rangle ^{|   {w}|
} \cdot\nabla R_z^V    \left\langle
x\right\rangle ^{{w}}u  \|   _{L_{{w}}^{2}(\RE^3)}\,.
\end{align*}
If $|   {w}|   \in\left[  0,1\right]  $, the functions
$\Delta\left\langle x\right\rangle ^{|   {w}|   }$ and
$\nabla\left\langle x\right\rangle ^{|   {w}|   }$ are bounded
and smooth; then the functions $V   R_z^V     \Delta\left\langle
x\right\rangle ^{|   {w}|   }  R_z^V   $
and $V   R_z^V     \nabla\left\langle x\right\rangle
^{|   {w}|   }  \cdot\nabla R_z^V   $
define bounded maps in $\B\left(  L^{2}(\RE^3)  ,H_{\sigma
}^{2}(\RE^3)  \right)  $ for any $\sigma\in\mathbb{R}$
(since $V$ has compact support). In this case, we get%
\begin{equation}
\|   V   R_z^V   u\|   _{L_{{w}}^{2}(\RE^3)}\leq c\,\|   \left\langle x\right\rangle
^{{w}}u\|   _{L^{2}(\RE^3)  }=c\,\|   u\|   _{L_{{w}}^{2}(\RE^3)  }\,,
\end{equation}
and as before, we obtain (\ref{R_V_weight_est}) from (\ref{R_V_est_1}). This
result and an induction argument on $|   {w}|   \in\left[
n,n+1\right]  $, allow to conclude the proof.
\end{proof}
\begin{lemma}
Let $V\in L^{2}_{comp}(\RE^{3})$ and ${w}>1/2$. Then, for all $k^{2}>0$,
\begin{equation}
\|   V   R_{ -k^{2}}^{V,\pm} u\|   _{L_{{w}}%
^{2}(\RE^3)  }\leq c\,\|
V   \|   _{L^{2}(\RE^3)  }\|
\,u\|   _{L_{{w}}^{2}(\RE^3)  }\,.
\label{V_bound_lim}%
\end{equation}
\end{lemma}
\begin{proof}
According to our assumptions, it results%
\[
\|   V   R_{ -k^{2}}^{V,\pm}  u\|   _{L_{{w}}%
^{2}(\RE^3)  }\leq c\,\|
V   \|   _{L^{2}(\RE^3)  }\|  R_{ -k^{2}}^{V,\pm}  u\|   _{L^{\infty}\left(  B_{R}\right)  }\,,
\]
and the injection $H^{2}\left(  B_{R}\right)  \hookrightarrow L^{\infty
}\left(  B_{R}\right)  $ yields%
\begin{equation}
\|   V   R_{ -k^{2}}^{V,\pm}   u\|   _{L_{{w}}%
^{2}(\RE^3)  }\leq c\,\|
V   \|   _{L^{2}(\RE^3)  }\| R_{ -k^{2}}^{V,\pm}   u\|   _{H^{2}\left(  B_{R}\right)  }\,.
\label{V_bound_lim_1}%
\end{equation}
Since $R_{ -k^{2}}^{V,\pm}   \in \B\left(  L_{{w}}^{2}(\RE^{3}),H_{-{w}}^{2}(\RE^3)  \right)
$, the inequalities%
\begin{equation}
\|   R_{ -k^{2}}^{V,\pm}   u\|   _{H^{2}(B_{R})  }\leq c\,\| R_{ -k^{2}}^{V,\pm} u\|   _{H_{-{w}}^{2}(\RE^3)  }\leq c\,
\|   u\|   _{L_{{w}}^{2}}(\RE^{3})\,, \label{V_bound_lim_est}%
\end{equation}
hold for ${w}>1/2$. Then the statement follows from (\ref{V_bound_lim_1}) and
(\ref{V_bound_lim_est}).
\end{proof}
This result yields the following mapping properties.
\begin{lemma}
\label{BenArtz} Let $V\in L^{2}_{comp}(\RE^{3})$ and ${w}>1$. For all compact subsets $K\subset\left(
0,+\infty\right)  $ there exists $c_{K}>0$ such that, for all $k^{2}\in K$ and for all $u\in L_{{w}}^{2}(\RE^3)  \cap\ker(  R^{V,+}_{-k^{2}}-R^{V,-}_{-k^{2}} )$,
\begin{equation}
\|   R_{ -k^{2}}^{V,\pm}  u\|   _{H^{2}(\RE^3)}\leq c_{K}\|   u\|   _{L_{{w}}%
^{2}(\RE^3)  }\,. \label{lap_BenArtzi_cond}%
\end{equation}
\end{lemma}
\begin{proof}
If $V=0$ the statement follows from \cite[Corollary 5.7(b)]{BeDe}; in this
case for all $k^{2}\in K$ and $u\in L_{{w}}^{2}(\RE^3)  \cap\ker(  R^{0,+}_{-k^{2}}-R^{0,-}_{-k^{2}} )$,
\begin{equation}\label{1}
\|   R_{ -k^{2}}^{0,\pm}  u\|   _{H^{2}(\RE^3)}\leq \t c_{K}\|   u\|   _{L_{{w}}%
^{2}(\RE^3)  }\,.
\end{equation}
for a suitable $\tilde{c}_{K}>0$ depending on $K$. From the identity
(\ref{Rpm}) there follows%
\begin{equation}
R_{ -k^{2}}^{V,\pm}  u=R_{ -k^{2}}^{0,\pm}  \big(
1-V   R_{ -k^{2}}^{V,\pm}  \big)  u\,, \label{lap_V_1_1}%
\end{equation}
and%
\begin{equation}
\ker\big(  R_{-k^{2}}^{0,+}  -R_{-k^{2}}^{0,-}\big)  \subseteq\big(  \uno-V   R_{-k^{2}}^{V,\pm}\big)
\ker\big(  R_{-k^{2}}^{V,+}  -R_{-k^{2}}^{V,-}\big) \,.
\end{equation}
Let $u\in L_{{w}}^{2}(\RE^3)  \cap\ker\big(  R_{-k^{2}}^{V,+}  -R_{-k^{2}}^{V,-}\big)$; then%
\begin{equation}
f=\big(  \uno-V   R_{-k^{2}}^{V,\pm}\big)  u\in\ker\big(  R_{-k^{2}}^{0,+}  -R_{-k^{2}}^{0,-}\big)\,,
\end{equation}
and using (\ref{V_bound_lim}) there follows%
\begin{equation}
f=\big(  \uno-V   R_{-k^{2}}^{V,\pm}\big)   u\in L_{{w}}%
^{2}(\RE^3)  \cap\ker\big(  R_{-k^{2}}^{0,+}  -R_{-k^{2}}^{0,-}\big) \,,
\end{equation}
with%
\begin{equation}
\|   f\|   _{L_{{w}}^{2}(\RE^3)  }%
\leq\big(  1+c\,\|   V   \|   _{L^{2}(\RE^{3})}\big)  \|   u\|   _{L_{{w}}^{2}(\RE^{3})}\,.
\end{equation}
Hence, from the representation (\ref{lap_V_1_1}) and the estimates
(\ref{1}), we finally obtain%
\begin{equation}
\|   R_{ -k^{2}}^{V,\pm}  u\|   _{H^{2}(\RE^{3}) }=\|   R_{ -k^{2}}^{0,\pm}
f\|   _{H^{2}(\RE^3)  }\leq\tilde{c}%
_{K}\|   f\|   _{L_{{w}}^{2}(\RE^3)  }\leq
c_{K}\|   u\|   _{L_{{w}}^{2}(\RE^3)  }\,.
\end{equation}
\end{proof}
The existence of the resolvent's limits on the continuous spectrum has been
discussed in \cite{Ren1} for a wide class of operators including
singular perturbations. In the particular case of a singularly perturbed
Laplacian described through the general formalism introduced in \cite{JDE},
a limiting absorption principle has
been given in \cite{JST}. In what follows, we use the same strategy used in
these works to establish a limiting absorption principle for the self-adjoint operators given in Theorem \ref{delta}.
\begin{theorem}\label{LAPdelta} Let $V\in L^{2}_{comp}(\RE^{3})$ and let $A_{V,\alpha}$ defined as in Theorem \ref{delta}. Then the limits%
\begin{equation}
R_{-k^{2}}^{V,\alpha,\pm} :=\lim_{\epsilon\downarrow0}\,\left(
-A_{V,\alpha}-\left(  k^{2}\pm i\varepsilon\right)  \right)  ^{-1}\,,
\label{LAP}%
\end{equation}
exist in $\B(  L_{w}^{2}(\RE^{3})  ,L_{-w}%
^{2}( \RE^{3}))  $ for all $w>1/2$ and
$k\in\mathbb{R}\backslash\left\{  0\right\}  $.
\end{theorem}
\begin{proof}  According to Theorem (\ref{LAPV}), the limits $R_{-k^{2}}^{V,\pm} $
exists for all $k^{2}>0$ and $w>1/2$ in the uniform operator topology of
$\B(  L_{w}^{2}(\RE^{3}))  ,H_{-w}%
^{2}(\RE^{3}))  $. Hence we follow, mutatis mutandis, the same arguments as in the proof on Theorem 4.1 in \cite{JST} (corresponding to the case $V=0$) to which we refer for more details: by \cite[Theorem 3.5
and Proposition 4.2]{Ren1}, our statement holds whenever there exist $c_{1}$,
$c_{2}$ and $c_{K}>0$ (the last constant depending on $K\subset(0,+\infty)$
compact), such that the following conditions are fulfilled:%
\begin{equation}
\forall
\,\sigma\in\mathbb{R}\,,\ \forall\,z\in\CO\backslash\{
\operatorname{Re}(z)>c_{1}\}\,,\qquad  R_{z}^{V}  ,\,R_{z}^{V,\alpha}  \in\B(L_{\sigma}^{2}(\RE^3))  \,, \label{LAP1}%
\end{equation}%
\begin{equation}
R_{z}^{V}  - R_{z}^{V,\alpha}  \in{\B}%
_{\infty}\left(  L^{2}(\RE^3) ,L_{\sigma}^{2}(\RE^3)\right)  \,,\qquad\sigma>2\,,\ z\in\{  \operatorname{Re}(z)>c_{2}\}
\label{LAP2}%
\end{equation}
(here ${\B}_{\infty}(  L^{2}(\RE^3)
,L_{\sigma}^{2}(\RE^3))  $ denotes the space of compact
operators from $L^{2}(\RE^3)$ to $L_{\sigma}%
^{2}(\RE^3) $), and, for all compact subset $K\subset(0,+\infty)$,
\begin{equation}
\forall\,k^{2}\in K\,,\ \forall
\,u\in L_{2w}^{2}(\RE^3) \cap\ker(R_{-k^{2}}^{V,+}-R_{-k^{2}}^{V,-})\,,
\quad
\| R_{-k^{2}}^{V,\pm}  u\|_{L^{2}(\RE^3)}\leq c_{K}\left\Vert u\right\Vert _{L_{2w}%
^{2}(\RE^3) }\,.
\label{LAP3}%
\end{equation}
Recalling that $A_{V}$ is bounded from above, there
exists $c_{1}>0$ such that $z\in\rho(  A_{V})  $ whenever $\operatorname{Re}%
(z)>c_{1}$; hence (\ref{LAP1})
holds for $R_{z}^{V}$ by \eqref{R_V_weight_est}. Since $\Gamma$ is compact, by \eqref{R_V_weight_est} and by the mapping properties of $\gamma_{0}$, one has $\gamma_{0}R^{V}_{z}\in\B(L^{2}_{\sigma}(\RE^{3}),H^{1}(\Gamma))$ and, by duality, $\SL^{V}_{z}\in
\B(H^{-1}(\Gamma),L^{2}_{-\sigma}(\RE^{3}))$. Thus, formula (\ref{krein})
gives (\ref{LAP1}) for  $R_{z}^{V,\alpha}  $.
\par
Since $(  1+\alpha\gamma_{0}\SL_{z}^{V})  ^{-1}\alpha\in\B_{\infty
}(  H^{s}(  \Gamma)  ,H^{-s}(  \Gamma))$, $0<s<1/2$, by the compact embeddings $H^{1}(\Gamma)\hookrightarrow H^{s}(\Gamma)$ and $H^{-s}(\Gamma)\hookrightarrow H^{-1}(\Gamma)$, one has $(  1+\alpha\gamma_{0}\SL_{z}^{V})  ^{-1}\alpha\in\B_{\infty
}(  H^{1}(  \Gamma)  ,H^{-1}(  \Gamma))$. So, since $\gamma_{0}R^{V}_{z}\in\B(L^{2}(\RE^{3}),H^{1}(\Gamma))$ and $\SL^{V}_{z}\in
\B(H^{-1}(\Gamma),L^{2}_{\sigma}(\RE^{3}))$,  (\ref{LAP2}) follows from (\ref{krein}). Finally, the condition (\ref{LAP3}) holds as a consequence of the Lemma \ref{BenArtz}.
\end{proof}
The previous results also allow to prove that the resolvent formula (\ref{krein}) survives in the limits
$z\to-(  k^{2}\pm i0)  $.
\begin{theorem}\label{Krein_lap} Let $V\in L^{2}_{comp}(\RE^{3})$, $k\in\RE\backslash\{0\}$ and let $A_{V,\alpha}$ defined as in Theorem \ref{delta}.
For any $w>{\frac12}$, the  limits
\be\label{G1}
\SL^{V,\pm}_{-k^{2}}:=\lim_{\epsilon\downarrow 0} \SL^{V,\pm}_{-(k^{2}\pm i\epsilon)}
\ee
exist in $\B(H^{-{s}}(  \Gamma) ,H^{\frac32-s}_{-w}(
\mathbb{R}^{3}) )$, $0<s\le 1$, and
\be\label{G2}
\SL^{V,\pm}_{-k^{2}}=\SL^{V}_{z}+(z+k^{2})R^{V,\pm}_{-k^{2}}\SL^{V}_{z}\,,\quad z\in\rho(A_{V})\,,
\ee
\be\label{G3}
(\SL^{V,\pm}_{-k^{2}})^*=\gamma_{0} R^{V,\mp}_{-k^{2}}\,.
\ee
The function $\SL^{V,\pm}_{-k^{2}}\xi$  solves, in the distribution space ${\D}'(\RE^{3}\backslash\Gamma)$ and for any $\xi\in H^{-{1}}(  \Gamma)$, the equation
$$
(\Delta-V+k^{2})\SL^{V,\pm}_{-k^{2}}\xi=0
$$
and there exist $c^{\pm}_{k^{2}}>0$ such that
\begin{equation}
\|\SL_{-k^{2}}^{V,\pm }\xi\|_{H^{3/2-s}_{-w}( \mathbb{R}^{3})
}\ge c_{k^{2}}^{\pm}\|\xi\|_{H^{-{s}}(  \Gamma)}\,.
\label{G_z_bijection}%
\end{equation}
Moreover, the limits%
\begin{equation}
\lim_{\epsilon\downarrow0}\left(  \uno+\alpha\gamma_{0}\SL_{-(  k^{2}\pm
i\varepsilon)}^{V}\right)  ^{-1}\!\!\alpha\,, \label{Weyl_lap}%
\end{equation}
exist in ${\B}(  H^{s}(  \Gamma)  ,H^{-s}(\Gamma))  $, $0<s<1/2$, and the operator $\uno+\alpha\gamma
_{0}\SL_{-k^{2}}^{V,\pm}  $ has a bounded inverse such that %
\begin{equation}
\left(  \uno+\alpha\gamma
_{0}\SL_{-k^{2}}^{V,\pm}\right)^{-1}\!\!\alpha=\lim
_{\epsilon\downarrow0}\left(  \uno+\alpha\gamma_{0}\SL_{-(  k^{2}\pm
i\varepsilon)}^{V,\pm}\right)  ^{-1}\!\!\alpha\,. \label{lap_krein_kernel}%
\end{equation}
Finally, the  limit resolvent $R_{-k^2}^{V,\alpha,\pm}  $ has the
representation%
\begin{equation}
R_{-k^2}^{V,\alpha,\pm}  =R_{-k^2}^{V,\pm}
-\SL_{-k^{2}}^{V,\pm}\left(  1+\alpha\gamma_{0}\SL_{-k^{2}}^{V,\pm}\right)
^{-1}\!\!\alpha\gamma_{0}R_{-k^2}^{V,\pm}  \,. \label{lap_krein}%
\end{equation}
%and the maps $z\mapsto R_{V,\alpha}\left(  z\right)  $ defined on
%$\mathbb{C}\backslash\mathbb{R}_{+}$ by%
%\begin{equation}
%z\mapsto R_{V,\alpha}^{\pm}\left(  z\right)  :=%
%\begin{cases}
%\left(  -A_{V,\alpha}+{z}\right)  ^{-1}\,, & z\in\mathbb{C}\backslash
%\mathbb{R}\,,\\
%R_{-k^2}^{V,\alpha,\pm}  \,, & k^{2}\in\mathbb{R}_{+}\,,
%\end{cases}
%\label{R_V_alpha_lap_id}%
%\end{equation}
%are respectively continuous on $\left(  \mathbb{C}^{\pm}\backslash
%\mathbb{R}_{+}\right)  \cup\mathbb{R}_{-}$ to $\mathsf{B}\left(  L_{\eta}%
%^{2}\left(  \RE^3 \right)  ,L_{-\eta}^{2}\left(  \mathbb{R}%
%^{n}\right)  \right)  $ with $\eta>1/2$.
\end{theorem}
\begin{proof}
The proof uses exactly the same argumentation of the proofs of Lemma 4.4 and Theorem 4.5 (which give the analogous results in the case $V=0$) provided in \cite{JST} and so is left to the reader.
\end{proof}
\end{section}
\begin{section} {Generalized eigenfunctions.}
We say that a function $u_{\pm}$ which solves, outside some large ball $B_{R}$, the Helmholtz equation $(\Delta+k^{2})u_{\pm}=0$, satisfies the $(\pm)$ {\it Sommerfeld radiation condition} (or {\it $u_{\pm}$ is $(\pm)$ radiating} for short) whenever
\be\label{src}
\lim_{|x|\to+\infty}\,|x|(\hat x\!\cdot\!\nabla \pm ik)u_\pm(x)=0
\ee
holds uniformly in $\hat x:=x/|x|$. \par Given $\psi^{0}_{k}\not=0$, a generalized free eigenfunction with eigenvalue $k^{2}\not=0$, i.e. $\psi^{0}_{k}\in H_{loc}^{2}(\RE^{3})$ and $(\Delta+k^{2})\psi^{0}_{k}=0$,  we say that  $\psi^{V,+/-}_{k}\not=0$ is an {\it incoming/outgoing   eigenfunction of $-A_{V}$ associated with the free wave} $\psi^{0}_{k}$ whenever $\psi^{V,\pm}_{k}\in H_{loc}^{2}(\RE^{3})$ solves $(\t A_{V}+k^{2})\psi^{V,\pm}_{k}=0$ and the {\it scattered field} $\psi^{V,\pm}_{k,sc} :=\psi^{V,\pm}_{k}-\psi^{0}_{k}$ satisfies the $(\pm)$ Sommerfeld radiation condition. Here
$\t A_{V}:H^{2}_{loc}(\RE^{3})\subset L^{2}_{loc}(\RE^{3})\to L^{2}_{loc}(\RE^{3})$, $V\in L^{2}_{comp}(\RE)$, denotes the broadening of $A_{V}$ defined by $\t A_{V}\psi:=\Delta\phi-V\psi$.
Let us notice that $\psi^{V,\pm}_{k,sc}$ satisfies the Helmholtz equation outside the support of $V$.
\par
The next result is a consequence of LAP for $A_{V}$:
\begin{theorem}\label{genV} The unique incoming and outgoing   eigenfunctions of $-A_{V}$, $V\in L^{2}_{comp}(\RE^{3})$, associated with the free wave $\psi^{0}_{k}$, $k\not=0$, are given by
$$
\psi^{V,\pm}_{k}:= \psi^{0}_{k}-R_{-k^{2}}^{V,\pm}V\psi^{0}_{k}\,.
$$
\end{theorem}
\begin{proof} By definition, $\psi^{+/-}_{k}\in H^{2}_{loc}(\RE^{3})$ is an incoming/outgoing   eigenfunction of $-A_{V}$ associated with $\psi^{0}_{k}$ if and only if  $(\psi^{\pm}_{k}-\psi^{0}_{k})$ is a $(\pm)$ radiating solution of $(\t A_{V}+k^{2})u=V\psi^{0}_{k}$. Since the potential $V$ is compactly supported, such an equation has an unique $(\pm)$ radiating solution. Indeed, if $u_{1}$ and $u_{2}$ were two different solutions then $u:=u_{1}-u_{2}$ would be a radiating solution,  outside some large ball $B_{R}$ containing the support of $V$, of $(\Delta+k^{2})u=0$. Thus $u|B_{R}^{c}=0$. Then, by the unique continuation principle for $\t A_{V}$ (see \cite{JerKo}), one gets $u=0$ everywhere. By Theorem \ref{LAPV},
$\psi^{V,\pm}_{k,sc} :=-R_{-k^{2}}^{V,\pm}V\psi^{0}_{k}\in H^{2}_{-w}(\RE^{3})$ solve the equation $(\t A_{V}+k^{2})\psi^{V,\pm}_{k,sc} =V\psi^{0}_{k}$. Moreover, by Theorem \ref{LAPV},
$\psi^{V,\pm}_{k,sc} =R_{-k^{2}}^{0,\pm}V(1-R_{-k^{2}}^{V,\pm}V)\psi^{0}_{k}$. Since $V$ is   compactly supported, $\psi^{V,\pm}_{k,sc} $ is $(\pm)$ radiating by \cite[Lemma 7, Subsection 7d, Section 8, Chapter II]{DL}.
\end{proof}
Now we extend the previous result to $A_{V,\alpha}$. At first we introduce the following broadening of $A_{V,\alpha}$ to the larger space $L_{loc}^{2}(\RE^{3})$:
$$
\t{A}_{V,\alpha}:\dom(\t A_{V,\alpha})\subseteq L_{loc}^{2}(\RE^{3})\to L^{2}_{loc}(\RE^{3})\,,\quad
$$
\begin{align*}
\dom(\t A_{V,\alpha}):=\{\psi\in H_{loc}^{\frac32-s}(\RE^{3}):\psi+\SL_{\circ}^{V}\alpha\gamma_{0}\psi\in
H^{2}_{loc}(\RE^{3})\}\,,
\end{align*}
$$\t A_{V,\alpha}\psi:=A_{V}\psi-\gamma^{*}_{0}\alpha\gamma_{0}\psi\,.
$$
We say that $\psi^{V,\alpha,+/-}_{k}\not=0$ is an {\it incoming/outgoing   eigenfunction of $-A_{V,\alpha}$ associated with the free wave $\psi^{0}_{k}$}, whenever $\psi^{V,\alpha\pm}_{k}\in\dom(\t A_{V,\alpha})$ solves the equation $(\t A_{V,\alpha}+k^{2})\psi^{V,\alpha,\pm}_{k}=0$ and the {\it scattered field} $\psi^{V,\alpha\pm}_{k,sc}:=\psi^{V,\alpha\pm}_{k}-\psi^{V,\pm}_{k}$ is $(\pm)$ radiating, where $\psi^{V,+/-}_{k}$ is the unique incoming/outgoing   eigenfunction of $-A_{V}$ associated, according to Theorem \ref{genV}, with $\psi^{0}_{k}$. Let us notice that $\psi^{V,\alpha,\pm}_{k,sc}$ satisfies the Helmholtz equation outside the $\supp(V)\cup\Gamma$.
\begin{theorem}\label{gen} Suppose that $\RE^{3}\backslash\overline\Omega$ is connected and $V\in L^{2}_{comp}(\RE^{3})$; let $-A_{V,\alpha}$ be defined as in Theorem \ref{delta}. Then the unique incoming and outgoing   eigenfunctions of $-A_{V,\alpha}$ associated with the free wave $\psi^{0}_{k}$ are given by
$$
\psi^{V,\alpha,\pm}_{k}:= \psi^{V,\pm}_{k}-\SL_{-k^{2}}^{V,\pm}(\uno+\alpha\gamma_{0}\SL_{-k^{2}}^{V,\pm})^{-1}\alpha\gamma_{0}\psi^{V,\pm}_{k}\,,%\quad \psi^{V,\pm}_{k}:=\psi^{0}_{k}-R_{-k^{2}}^{V,\pm}V\psi^{0}_{k}
\,.
$$
\end{theorem}
\begin{proof} By our definitions, $\t\psi^{+/-}_{k}\in\dom(\t A_{V,\alpha})$ is an incoming/outgoing   eigenfunction of $-A_{V,\alpha}$ associated with $\psi^{0}_{k}$ if and only if  $(\t\psi^{\pm}_{k}-\psi^{V,\pm}_{k})$ is a $(\pm)$ radiating solution of $(A_{V}+k^{2})u-\gamma_{0}^{*}\alpha\gamma_{0}u=\gamma_{0}^{*}\alpha\gamma_{0}\psi^{V,\pm}_{k}$ belonging to $H^{\frac32-s}_{loc}(\RE^{3})$. Since both the potential $V$ and the distribution $\gamma_{0}^{*}\xi$ are compactly supported, such an equation as an unique $(\pm)$ radiating solution. Indeed, if $u_{1}$ and $u_{2}$ were two different solutions then $u:=u_{1}-u_{2}$ would be a radiating solution,  outside some large ball $B_{R}$ containing both $\supp(V)$ and $\supp(\gamma_{0}^{*}\xi)$, of $(\Delta+k^{2})u=0$. Thus $u|B_{R}^{c}=0$. Then, by the unique continuation principle for $A_{V}$ (see \cite{JerKo}), one gets $u|\RE^{3}\backslash\overline\Omega=0$.  Since $u\in H^{\frac32-s}_{loc}(\RE^{3})$, then $\gamma_{0}u=0$ and so $u$ is a radiating solution of $(A_{V}+k^{2})u=0$; thus, proceeding as in the proof of Theorem \ref{genV}, $u=0$ everywhere. To conclude the proof we need to show that $\psi^{V,\alpha,\pm}_{k}\in\dom(\t A_{V,\alpha})$, i.e. that $\psi_{\circ}:=\psi^{V,\alpha,\pm}_{k}+\SL_{\circ}^{V}\alpha\gamma_{0}\psi^{V,\alpha,\pm}_{k}\in H^{2}_{loc}(\RE^{3})$, that $(\t A_{V,\alpha}+k^{2})\psi^{V,\alpha,\pm}_{k}=0$ and that $\SL_{-k^{2}}^{V,\pm}(\uno+\alpha\gamma_{0}\SL_{-k^{2}}^{V,\pm})^{-1}\alpha\gamma_{0}\psi^{V,\pm}_{k}$ in $(\pm)$ radiating. Since $\alpha\gamma_{0}\psi^{V,\alpha,\pm}_{k}=(\uno+\alpha\gamma_{0}\SL_{-k^{2}}^{V,\pm})^{-1}\alpha\gamma_{0}\psi^{V,\pm}_{k}$, one has, by \eqref{G2},
$$
\psi_{\circ}=\psi^{V,\pm}_{k}+(\SL_{\circ}^{V}-\SL_{-k^{2}}^{V,\pm})\alpha\gamma_{0}\psi^{V,\alpha,\pm}_{k}=\psi^{V,\pm}_{k}-(\lambda_{\circ}+k^{2})R_{-k^{2}}^{V,\pm}\SL_{\circ}^{V}\alpha\gamma_{0}\psi^{V,\alpha,\pm}_{k}\in H^{2}_{loc}(\RE^{3})\,.
$$
Then
\begin{align*}
(\t A_{V,\alpha}+k^{2})\psi^{V,\alpha,\pm}_{k}=&
(A_{V,\alpha}+k^{2})(\psi^{V,\pm}_{k}-\SL_{-k^{2}}^{V,\pm}\alpha\gamma_{0}\psi^{V,\alpha,\pm}_{k})-\gamma_{0}^{*}\alpha\gamma_{0}\psi^{V,\alpha,\pm}_{k}\\
=&(-A_{V}-k^{2})R_{-k^{2}}^{V,\pm}\gamma_{0}^{*}\alpha\gamma_{0}\psi^{V,\alpha,\pm}_{k}
-\gamma_{0}^{*}\alpha\gamma_{0}\psi^{V,\alpha,\pm}_{k}=0\,.
\end{align*}
Finally, by \eqref{Rpm} and by \eqref{G3}, $\SL_{-k^{2}}^{V,\pm}\xi=R_{-k^{2}}^{0,\pm}(\gamma^{*}_{0}\xi-V\SL_{-k^{2}}^{V,\pm}\xi)$ and so, since both $\gamma_{0}^{*}\xi$ and $V$ are compactly supported, $\SL_{-k^{2}}^{V,\pm}\xi$ is $(\pm)$ radiating by \cite[Lemma 7, Subsection 7d, Section 8, Chapter II]{DL}.
\end{proof}
\begin{remark}\label{gamma1}
By the resolvent identity $R_{z}^{V}=R_{z}^{0}-R_{z}^{V}VR_{z}^{0}$ and by \eqref{G2}, one gets
$\SL^{V,\pm}_{-k^{2}}\xi=\SL^{0}_{z}\xi+\phi_{\xi}$, where $\phi_{\xi}\in H_{-w}^{2}(\RE^{3})$. Thus, by $[\hat\gamma_{1}]\SL^{0}_{z}\xi=-\xi$ (see \cite[Theorem 6.11]{McLe}) and $H_{-w}^{2}(\RE^{3})\subset\ker[\hat\gamma_{1}]$, one obtains  $$[\hat\gamma_{1}]\SL^{V,\pm}_{-k^{2}}\xi=-\xi\,.
$$
Then, by the identity $\alpha\gamma_{0}\psi^{V,\alpha,\pm}_{k}=(\uno+\alpha\gamma_{0}\SL_{-k^{2}}^{V,\pm})^{-1}\alpha\gamma_{0}\psi^{V,\pm}_{k}$, one obtains the relations
$$
\alpha\gamma_{0}\psi^{V,\alpha,\pm}_{k}=[\hat\gamma_{1}]\psi^{V,\alpha,\pm}_{k}\,,\qquad\psi^{V,\alpha,\pm}_{k}= \psi^{V,\pm}_{k}+\SL_{-k^{2}}^{V,\pm}[\hat\gamma_{1}]\psi^{V,\alpha,\pm}_{k}\,.
$$
\end{remark}
For any $k>0$, we define the set $$\Sigma_{k}:=\{\rho\in\CO^{3}: \rho\!\cdot\!\rho=-k^{2}\}\,,$$ where $\cdot$ denotes the euclidean scalar product, equivalently
$$\Sigma_{k}=\{\rho=w\hat\zeta+i\sqrt{w^{2}+k^{2}}\,\hat\xi: {w}\ge 0,\ (\hat\zeta,\hat\xi)\in\Sf^{2}\times \Sf^{2},\   \hat\zeta\!\cdot\!\hat\xi=0\}\,.$$ Clearly any function of the kind $\psi_{\rho}(x):=e^{\rho\cdot x}$, $\rho\in \Sigma_{k}$, is a generalized eigenfunction of $-\Delta$ with eigenvalue $k^{2}$.
\begin{corollary}\label{ff} Given a ball $B_{R_{\circ}}\supset\Omega\cup\supp(V)$, the outgoing   eigenfunction $\psi^{V,\alpha}_{\rho}$ associated, according to Theorem \ref{gen},  with $\psi_{\rho}(x):=e^{\rho\cdot x}$, $\rho\in \Sigma_{k}$,
has the asymptotic behavior
$$
\psi_{\rho}^{V,\alpha}(x)=e^{\rho\cdot x}+\frac{e^{ik|x|}}{|x|}\, \psi^{\infty}_{V,\alpha}(k,\rho,\hat x)+O(|x|^{-2})\,, \quad |x|\gg R_{\circ}\,,
$$
uniformly in all directions $\hat x:=\frac{x}{|x|}$. Moreover
\be\label{int}
\psi^{\infty}_{V,\alpha}(k,\rho,\hat x)=\frac1{4\pi}\int_{\partial B_{R_{\circ}}}\left(\psi_{\rho,sc}^{V,\alpha} (y)\,\hat y\!\cdot\!\nabla e^{-ik\hat x\cdot y}-e^{-ik\hat x\cdot y}\,\hat y\!\cdot\!\nabla \psi_{\rho,sc}^{V,\alpha} (y)\right) d\sigma_{R_{\circ}}(y)\,,
\ee
where
$$
\psi_{\rho,sc}^{V,\alpha} =-R^{V,-}_{-k^{2}}V\psi_{\rho}-\SL^{V,-}_{-k^{2}}\alpha\gamma_{0}\psi^{V,\alpha}_{\rho}\,.
$$
\end{corollary}
\begin{proof} By Theorems \ref{genV} and \ref{gen}, and by the identity $\alpha\gamma_{0}\psi^{V,\alpha}_{\rho}=(\uno+\alpha\gamma_{0}\SL_{-k^{2}}^{V,-})^{-1}\alpha\gamma_{0}\psi^{V}_{\rho}$, where $\psi^{V}_{\rho}$ denotes the outgoing   eigenfunction associated, according to Theorem \ref{genV},  with $\psi_{\rho}$, one has $\psi^{V,\alpha}_{\rho}=\psi_{\rho}+\psi_{\rho,sc}^{V,\alpha}$.  Since $(-\Delta+k^{2})\psi_{\rho,sc}^{V,\alpha} =0$ outside $\Omega\cup\supp(V)$, the thesis is consequence of the asymptotic representation of the radiating solutions of the Helmholtz equation (see e.g. \cite[Theorem 2.6]{CK}).
\end{proof}
\begin{remark} Since $(-\Delta+k^{2})\psi_{\rho,sc}^{V,\alpha} =0$ outside $\Omega\cup\supp(V)$, by elliptic regularity $\psi_{\rho,sc}^{V,\alpha}$ is smooth outside $\Omega\cup\supp(V)$ and so relation \eqref{int} is well defined.
\end{remark}
According to Corollary \ref{ff}, the scattering amplitude $s_{V,\alpha}$ for the Schr\"odinger operator $A_{V,\alpha}$ is then related to the {\it far-field pattern} $\psi^{\infty}_{V,\alpha}$ by the simple relation
\be\label{sr}
s_{V,\alpha}(k,\hat\xi,\hat x)=(2\pi)^{\frac32}\,\psi^{\infty}_{V,\alpha}(k,ik\hat\xi,,\hat x)\,.
\ee
Indeed, by Corollary \ref{ff}, the outgoing   eigenfunction $\psi^{V,\alpha}_{k}$ associated, according to Theorem \ref{gen},  with $\psi^{0}_{k}(x):=e^{ik\hat\xi\cdot x}$,
has the asymptotic behavior
$$
\psi_{k}^{V,\alpha}(x)=e^{ik\hat\xi\cdot x}+\frac1{(2\pi)^{3/2}}\,\frac{e^{ik|x|}}{|x|}\, s_{V,\alpha}(k,\hat\xi,\hat x)+O(|x|^{-2})\,, \quad |x|\gg R_{\circ}\,.
$$
The next lemma shows that the scattering amplitude univocally determines both the far-field $\psi^{\infty}_{V,\alpha}$ and the scattered field $\psi^{V,\alpha}_{\rho,sc}$.
\begin{remark}
Here and below, when considering two different self-adjoint operators $A_{V_{1},\alpha_{1}}$ and $A_{V_{2},\alpha_{2}}$ we mean that they can be eventually be defined in terms of  two different subset $\Omega_{1}$ and $\Omega_{2}$, so that  $(\partial\Omega_{1}=)\Gamma_{1}\not=\Gamma_{2}(=\partial\Omega_{2})$ is allowed.
\end{remark}
\begin{lemma}\label{ampl} Under the same hypotheses as inTheorem \ref{gen}, suppose that, for some $k>0$,  $$
s_{V_{1},\alpha_{1}}(k,\hat\xi,\hat x)=s_{V_{2},\alpha_{2}}(k,\hat\xi,\hat x)
\quad\text{for all $(\hat\xi,\hat x)\in \Sf^{2}\times \Sf^{2}$.}$$
Then
\be\label{eqff}
\psi^{\infty}_{V_{1},\alpha_{1}}(k,\rho,\hat x)=\psi^{\infty}_{V_{2},\alpha_{2}}(k,\rho,\hat x)\,,\quad\text{for all $(\rho,\hat x)\in \Sigma_{k}\times \Sf^{2}$}
\ee
and
\be\label{eqsc}
\psi^{V_{1},\alpha_{1}}_{\rho,sc}(x)=\psi^{V_{2},\alpha_{2}}_{\rho,sc}(x)\,,
\quad\text{for all $\rho\in \Sigma_{k}$ and for any $x\in B_{R_{\circ}}^{c}$,}
\ee
where $B_{R_{\circ}}\supset (\Omega_{1}\cup\Omega_{2}\cup\supp(V_{1})\cup\supp(V_{2}))$.
\end{lemma}
\begin{proof} By \eqref{sr} and \eqref{int}, to get \eqref{eqff} it suffices to show that, for some $R>R_{\circ}$, if $\psi_{ik\hat\xi,sc}^{V_{1},\alpha_{1}}|B_{R}=\psi_{ik\hat\xi,sc}^{V_{2},\alpha_{2}}|B_{R}$,  for all $\hat\xi\in\Sf^{2}$  then  $\psi_{\rho,sc}^{V_{1},\alpha_{1}}|B_{R}=\psi_{\rho,sc}^{V_{2},\alpha_{2}}|B_{R}$ for all $\rho\in\Sigma_{k}$.\par
Since $C(\Sf^{2})$ is dense in $L^{2}(\Sf^{2})$, according to \cite[Theorem 2]{weck} there exists a sequence $\{f_{n}\}_{1}^{\infty}\subset C(\Sf^{2})$
such that $H_{k}f_{n}\to \psi_{\rho}$ in $H^{2}(B_{R})$, where $H_{k}$, $k\not=0$, is the Herglotz operator
$$
H_{k}f(x):=\int_{\Sf^{2}}f(\hat\xi)\,e^{ik\hat\xi\cdot x}\,d\sigma(\xi)\,.
$$
Writing the above integral as a limit of a Riemann's sum, $\psi_{\rho}$ can be obtained as a $H^{2}(B_{R})$-limit of a sequence of functions of the kind $\sum_{m=1}^{n}a_{m,n}\, e^{ik\hat\xi_{m,n}\cdot x}$. Since, by Theorems \ref{genV} and \ref{gen},
$$\psi^{V,\alpha}_{\rho,sc}=-L^{V,\alpha}_{-k^{2}}\,\psi_{\rho}\,,\quad L^{V,\alpha}_{-k^{2}}:=R^{V,-}_{-k^{2}}V+\SL^{V,-}_{-k^{2}}
(\uno+\alpha\gamma_{0}\SL_{-k^{2}}^{V,-})^{-1}\alpha\gamma_{0}(\uno-R^{V,-}_{-k^{2}}V)\,,
$$
to get \eqref{eqff} one needs to show that the linear operator $L^{V,\alpha}_{-k^{2}}$ is continuous on $H_{loc}^{2}(\RE^{3})$ to $L^{2}_{-w}(\RE^{3})$. Since $V\in L^{2}_{comp}(\RE^{3})$, the multiplication operator associated with $V$ belongs to $\B(H^{2}(B_{R}), L^{2}_{w}(\RE^{3}))$; thus, by Theorem \ref{LAPV}, $R^{V,-}_{-k^{2}}V\in\B(H^{2}(B_{R}), H^{2}_{-w}(\RE^{3}))$. Moreover, by Theorem \ref{Krein_lap}, $\SL^{V,-}_{-k^{2}}
(\uno+\alpha\gamma_{0}\SL_{-k^{2}}^{V,-})^{-1}\alpha\gamma_{0}\in \B(H^{\frac12+s}(B_{R}), H^{\frac32-s}_{-w}(\RE^{3}))$. So $L^{V,\alpha}_{-k^{2}}\in \B(H^{2}(B_{R}), H^{\frac32-s}_{-w}(\RE^{3}))$ and \eqref{eqff} holds true.\par
If $\psi^{\infty}_{V_{1},\alpha_{1}}(k,\rho,\hat x)=\psi^{\infty}_{V_{2},\alpha_{2}}(k,\rho,\hat x)$, then, by Corollary \ref{ff},   $u_{sc}(x):=\psi^{V_{1},\alpha_{1}}_{\rho,sc}(x)-\psi^{V_{2},\alpha_{2}}_{\rho,sc}(x)=O(|x|^{-2})$. Since $u_{sc}$ solves the Helmoltz equation $(\Delta+k^{2})u_{sc}=0$ outside $B_{R_{\circ}}$, by Rellich's lemma (see e.g. \cite[Theorem 2.14]{CK}), one gets $u_{sc}|B_{R_{\circ}}=0$, i.e. \eqref{eqsc}.
\end{proof}
\end{section}
\begin{section}{Uniqueness  in inverse Schr\"odinger scattering.}
Given $V\in L_{comp}^{2}(\RE^{3})$ and $\alpha\in H^{-s}(\Gamma)$, $0<s\le 1$, let us define
$\t V_{\alpha}\in H_{comp}^{-s-\frac12}(\RE^{3})$, by $\t V_{\alpha}:=V+\gamma_{0}^{*}\alpha$.
For any $\psi\in H_{loc}^{s+\frac12}(\RE^{3})$ the product $\psi\t V_{\alpha}\in \D'(\RE^{3})$ is well defined and, by Lemma \ref{gamma}, $\psi\t V_{\alpha}=V\psi+\gamma^{*}_{0}\alpha\gamma_{0}\psi$. Thus, since $M(H^{s}(\Gamma),H^{-s}(\Gamma))\subseteq H^{-s}(\Gamma)$, we have
$$
\langle\gamma^{*}_{0}\alpha\gamma_{0}\psi,\phi\rangle_{\D',\D}=
\langle\alpha\gamma_{0}\psi,\gamma_{0}\phi\rangle_{H^{-s},H^{s}}\,,\quad \psi,\phi\in H^{s+\frac12}(\RE^{3})
$$
and so $\psi\t V_{\alpha}\in H_{comp}^{-s-\frac12}(\RE^{3})$ whenever $\alpha\in M(H^{s}(\Gamma),H^{-s}(\Gamma))$ and $\psi\in H_{loc}^{s+\frac12}(\RE^{3})$. In particular, whenever $V$ and $\alpha$ are as in the definition of $\t A_{V,\alpha}$ and $\psi\in\dom(\t A_{V,\alpha})$, one has $\psi\t V_{\alpha}\in H_{comp}^{-1}(\RE^{3})$.\par
We need a preparatory lemma before stating the main result in this section:
\begin{lemma}\label{prep} Let $\psi^{V_{1},\alpha_{1}}_{\rho_{1}}$ and $\psi^{V_{2},\alpha_{2}}_{\rho_{2}}$  be the outgoing eigenfunctions of $A_{V_{1},\alpha_{1}}$ and $A_{V_{2},\alpha_{2}}$, associated, according to Theorem \ref{gen}, with $\psi_{\rho_{1}}(x)=e^{\rho_{1}\cdot x}$ and $\psi_{\rho_{2}}(x)=e^{\rho_{2}\cdot x}$, where both $\rho_{1}$ and $\rho_{2}$ belong to $\Sigma_{k}$. Then
$$
s_{V_{1},\alpha_{1}}(k,\cdot,\cdot)=s_{V_{2},\alpha_{2}}(k,\cdot,\cdot)\quad\Longrightarrow\quad
\langle \psi^{V_{1},\alpha_{1}}_{\rho_{1}}(\t V_{\alpha_{1}}-\t V_{\alpha_{2}}),\psi^{V_{2},\alpha_{2}}_{\rho_{2}}\rangle_{H_{comp}^{-1},H_{loc}^{1}}
=0\,.$$
\end{lemma}
\begin{proof} By the definition of $\tilde V_{\alpha}$, by Lemma \ref{gamma} and by Remark \ref{gamma1}, one obtains
\be\begin{split}\label{Vt}
&\langle \psi^{V_{1},\alpha_{1}}_{\rho_{1}}(\t V_{\alpha_{1}}-\t V_{\alpha_{2}}),\psi^{V_{2},\alpha_{2}}_{\rho_{2}}\rangle_{H_{comp}^{-1},H_{loc}^{1}}\\
=&\langle \psi^{V_{1},\alpha_{1}}_{\rho_{1}}(V_{1}-V_{2}),\psi^{V_{2},\alpha_{2}}_{\rho_{2}}\rangle_{L^{2}(\RE^{3})}\\
&+\langle \alpha_{1}\gamma_{0,1}\psi^{V_{1},\alpha_{1}}_{\rho_{1}},\gamma_{0,1}\psi^{V_{2},\alpha_{2}}_{\rho_{2}}\rangle_{H^{-s}(\Gamma_{1}),H^{s}(\Gamma_{1})}
-\langle \alpha_{2}\gamma_{0,2}\psi^{V_{1},\alpha_{1}}_{\rho_{1}},\gamma_{0,2}\psi^{V_{2},\alpha_{2}}_{\rho_{2}}\rangle_{H^{-s}(\Gamma_{2}),H^{s}(\Gamma_{2})}\\
=&\langle \psi^{V_{1},\alpha_{1}}_{\rho_{1}}(V_{1}-V_{2}),\psi^{V_{2},\alpha_{2}}_{\rho_{2}}\rangle_{L^{2}(\RE^{3})}\\
&+\langle [\hat\gamma_{1,1}]\psi^{V_{1},\alpha_{1}}_{\rho_{1}},\gamma_{0,1}\psi^{V_{2},\alpha_{2}}_{\rho_{2}}\rangle_{H^{-s}(\Gamma_{1}),H^{s}(\Gamma_{1})}
-\langle [\hat\gamma_{1,2}]\psi^{V_{1},\alpha_{1}}_{\rho_{1}},\gamma_{0,2}\psi^{V_{2},\alpha_{2}}_{\rho_{2}}\rangle_{H^{-s}(\Gamma_{2}),H^{s}(\Gamma_{2})}\,,
\end{split}
\ee
where $\gamma_{0,m}$ and $[\hat\gamma_{1,m}]$ denote the trace operators on   $\Gamma_{m}$ and
the jump of the normal derivatives across $\Gamma_{m}$ respectively. \par
Let $\psi_{m,\rho}\,$, $m=1,2$, be outgoing   eigenfunctions of $\t A_{V_{m},\alpha_{m}}$ associated with $\psi_{\rho}$, $\rho\in\Sigma_{k}$. Setting
$$
\Omega_{\ell,m}:=\begin{cases}\Omega_{m}\,,&\ell=0\\
(\RE^{3}\backslash\overline\Omega_{m})\cap B_{R}\,,&\ell=1\,,
\end{cases}
$$
where $B_{R}\supset (\Omega_{1}\cup\Omega_{2}\cup\supp(V_{1})\cup\supp(V_{2}))$,
one gets
$$
\Delta_{\Omega_{\ell,m}}(\psi_{m,\rho}|\Omega_{\ell,m})=((V_{m}-k^{2})\psi_{m,\rho})|\Omega_{\ell,m}\,,
$$
so that
$$
\psi_{m,\rho}|\Omega_{\ell,m}\in H^{1}_{\Delta}(\Omega_{\ell,m}):=\{u\in H^{1}(\Omega_{\ell,m}):\Delta_{\Omega_{\ell,m}}u\in L^{2}(\Omega_{\ell,m})\}\,.
$$
Then, according to the half Green's formula (see \cite[Theorem 4.4]{McLe}), one obtains
\begin{align*}
&\langle-\Delta_{\Omega_{0,m}}(\psi_{m,\rho}|\Omega_{0,m}),(\psi_{n,\rho'}|\Omega_{0,m})\rangle_{L^{2}(\Omega_{0,m})}+
\langle-\Delta_{\Omega_{1,m}}(\psi_{m,\rho}|\Omega_{1,m}),\psi_{n,\rho'}|\Omega_{1,m}\rangle_{L^{2}(\Omega_{1,m})}\\
=&\langle\nabla\psi_{m,\rho},\nabla\psi_{n,\rho'}\rangle_{L^{2}(B_{R})}
+
\langle[\hat\gamma_{1,m}]\psi_{m,\rho},\gamma_{0,m}\psi_{n,\rho'}\rangle_{H^{-1/2}(\Gamma_{m}), H^{1/2}(\Gamma_{m})}\\&-
\langle\gamma_{1,R}\psi_{m,\rho},\gamma_{0,R}\psi_{n,\rho'}\rangle_{L^{2}(\partial B_{R})}\,,
\end{align*}
where $\gamma_{0,R}$ and $\gamma_{1,R}$ denote the trace operator and the normal derivative on  $\partial B_{R}$ respectively. Thus, since $\Delta\psi_{m,\rho}=(\t V_{m}-k^{2})\psi_{m,\rho}$,  by \eqref{Vt},  one gets
\begin{align*}
&\langle \psi^{V_{1},\alpha_{1}}_{\rho_{1}}(\t V_{\alpha_{1}}-\t V_{\alpha_{2}}),\psi^{V_{2},\alpha_{2}}_{\rho_{2}}\rangle_{H_{comp}^{-1},H_{loc}^{1}}
\\=&
\langle\gamma_{1,R}\psi^{V_{1},\alpha_{1}}_{\rho_{1}},\gamma_{0,R}\psi^{V_{2},\alpha_{2}}_{\rho_{2}}\rangle_{L^{2}(\partial B_{R})}-\langle\gamma_{0,R}\psi^{V_{1},\alpha_{1}}_{\rho_{1}},\gamma_{1,R}\psi^{V_{2},\alpha_{2}}_{\rho_{2}}\rangle_{L^{2}(\partial B_{R})}
\end{align*}
and, by \eqref{eqsc} in Lemma \ref{ampl},
\begin{align*}
0=&\langle \psi^{V_{1},\alpha_{1}}_{\rho_{1}}(\t V_{\alpha_{1}}-\t V_{\alpha_{1}}),\psi^{V_{1},\alpha_{1}}_{\rho_{2}}\rangle_{H_{comp}^{-1},H_{loc}^{1}}
\\=&
\langle\gamma_{1,R}\psi^{V_{1},\alpha_{1}}_{\rho_{1}},\gamma_{0,R}\psi^{V_{1},\alpha_{1}}_{\rho_{2}}\rangle_{L^{2}(\partial B_{R})}-\langle\gamma_{0,R}\psi^{V_{1},\alpha_{1}}_{\rho_{1}},\gamma_{1,R}\psi^{V_{1},\alpha_{1}}_{\rho_{2}}\rangle_{L^{2}(\partial B_{R})}
\\=&
\langle\gamma_{1,R}\psi^{V_{1},\alpha_{1}}_{\rho_{1}},\gamma_{0,R}\psi^{V_{2},\alpha_{2}}_{\rho_{2}}\rangle_{L^{2}(\partial B_{R})}-\langle\gamma_{0,R}\psi^{V_{1},\alpha_{1}}_{\rho_{1}},\gamma_{1,R}\psi^{V_{2},\alpha_{2}}_{\rho_{2}}\rangle_{L^{2}(\partial B_{R})}\,.
\end{align*}
\end{proof}
Finally we state our uniqueness result for inverse Schr\"odinger scattering:
\begin{theorem}\label{uniq} Let $V_{1},V_{2}\in L^{2}_{comp}(\RE^{3})$, $\alpha_{1},\in M(H^{s}(\Gamma_{1}), H^{-s}(\Gamma_{1}))$, $\alpha_{2}\in M(H^{s}(\Gamma_{2}), H^{-s}(\Gamma_{2}))$, $0<s<1/2$, and suppose that $\RE^{3}\backslash\overline\Omega_{1}$ and $\RE^{3}\backslash\overline\Omega_{2}$ are connected. Then
$$
s_{V_{1},\alpha_{1}}(k,\cdot,\cdot)=s_{V_{2},\alpha_{2}}(k,\cdot,\cdot)\quad\Longrightarrow\quad V_{1}=V_{2}\,,\ \supp(\alpha_{1})=\supp(\alpha_{2})\,,\ \alpha_{1}=\alpha_{2}\,.
$$
\end{theorem}
\begin{proof}  Let $\psi^{V_{m},\alpha_{m}}_{\rho_{m}}$, $m=1,2$, be as in Lemma \ref{prep}  and choose $\rho_{1}$ and $\rho_{2}$ in $\Sigma_{k}$ in such a way that $\bar\rho_{1}+\rho_{2}=-i \xi$, $\xi\in\RE^{3}$. Further set $\phi^{V_{m},\alpha_{m}}_{\rho_m}(x):=e^{-\rho\cdot x}\psi^{V_{m},\alpha_{m}}_{\rho_{m}}(x)-1$, so that $\psi^{V_{m},\alpha_{m}}_{\rho_m}(x)=e^{\rho\cdot x}(1+\phi^{V_{m},\alpha_{m}}_{\rho_{m}}(x))$. Then, by Lemma \ref{prep}, setting $u_{\xi}(x):=e^{-i\xi\cdot x}$,
\begin{align*}
F_{\xi}(\rho_{1},\rho_{2}):=&\langle \t V_{\alpha_{1}}-\t V_{\alpha_{2}},u_{\xi}(\phi^{V_{1},\alpha_{1}}_{\rho_{1}}+\phi^{V_{2},\alpha_{2}}_{\rho_{2}})\rangle_{H_{comp}^{-1},H_{loc}^{1}}+
\langle \phi^{V_{1},\alpha_{1}}_{\rho_{1}}(\t V_{\alpha_{1}}-\t V_{\alpha_{2}}),u_{\xi}\phi^{V_{2},\alpha_{2}}_{\rho_{2}}\rangle_{H_{comp}^{-1},H_{loc}^{1}}\\
&=\langle \t V_{\alpha_{2}}-\t V_{\alpha_{1}},u_{\xi}\rangle_{H_{comp}^{-1},H_{loc}^{1}}
=\widehat{\t V_{\alpha_2}}(\xi)-\widehat{\t V_{\alpha_1}}(\xi)
\,.
\end{align*}
By our definitions, setting $\psi_{\rho}(x)=e^{\rho\cdot x}$, one obtains
\begin{align*}
0=&(\t A_{V_{m},\alpha_{m}}+k^{2})\psi^{V,\alpha}_{\rho_{m}}=(\Delta+k^{2})(\psi_{\rho_{m}}(1+\phi^{V_{m},\alpha_{m}}_{\rho_{m}}))-\psi_{\rho_{m}}(1+\phi^{V_{m},\alpha_{m}}_{\rho_{m}})\t V_{\alpha_{m}}\\
=& \psi_{\rho_{m}}\big(\Delta \phi^{V_{m},\alpha_{m}}+2\rho_{m}\!\cdot\!\nabla\phi_{\rho_{m}}^{V_{m},\alpha_{m}}-
(1+\phi_{\rho_{m}}^{V_{m},\alpha_{m}})\t V_{\alpha_{m}}\big)\,.
\end{align*}
Thus $\phi^{V,\alpha}_{\rho_{m}}$ solves the equation
$$
\Delta \phi_{\rho_{m}}^{V_{m},\alpha_{m}}+2\rho_{m}\!\cdot\!\nabla\phi_{\rho_{m}}^{V_{m},\alpha_{m}}-\phi^{V_{m},\alpha_{m}}_{\rho_{m}}\t V_{\alpha_{m}}
=\t V_{\alpha_{m}}\,.
$$
Decaying estimates %(with respect to $|\rho|\to\infty$ and in appropriate function spaces)
of solutions of such an equation have been obtained in various papers concerning Calder\'on's uniqueness problem. In particular we use the recent results provided in \cite{Hab}.
Since $\alpha_{m}\in M(H^{s}(\Gamma),H^{-s}(\Gamma))\subseteq H^{-s}(\Gamma)$, $0<s<1/2$, one has  $\t V_{\alpha_{m}}\in H^{-1}_{comp}(\RE^{3})\subset W^{-1,3/2}_{comp}(\RE^{3})$.
Thus (see e.g. \cite[Theorem 3.12 and Corollary 3.23]{AF}) $\t V_{\alpha_{m}}=\sum_{i=1}^{3}\nabla_{\!i}f_{m,i}+h_{m}$ where  $f_{m,i}, h\in L^{3}(\RE^{3})$. Then, taking $\chi_{m}\in C_{comp}^{\infty}(\RE^{3})$ such that $\chi_{m}=1$ on $\supp(\t V_{\alpha_{m}})$, one has $\t V_{\alpha_{m}}=\chi_{m}\t V_{\alpha_{m}}=\sum_{i=1}^{3}\nabla_{\!i}(\chi_{m}f_{m,i})-\sum_{i=1}^{3}f_{m,i}\nabla_{i}\chi_{m}+\chi_{m}h_{m}=\sum_{i=1}^{3}\nabla_{\!i}\t f_{m,i}+\t h_{m}$ where  $\t f_{m,i}, \t h\in L_{comp}^{3}(\RE^{3})$. Therefore \cite[Theorem 5.3]{Hab}\footnote{Such a theorem is stated for  $q=\gamma^{-1/2}\Delta\gamma^{1/2}$, $\nabla\log\gamma\in L_{comp}^{3}(\RE^{3})$; however the proof only uses the decomposition $q=\sum_{i=1}^{3}\nabla_{i}f_{i}+h$ where  $f_{i}\in L^{3}_{comp}(\RE^{3})$ and $h\in L^{3/2}(\RE^{3})$.} applies to $\t V_{\alpha_{m}}$ and so, by the same reasoning as in the (second part) of the proof of Theorem 1.1 in \cite{Hab} (see in particular inequality (31)), for any $\xi\in\RE^{3}$ one gets the existence of two suitable sequences $\{\rho_{m,n}\}_{n=1}^{+\infty}\subset \Sigma_{k}$, $\overline\rho_{1,n}+\rho_{2,n}=-i \xi$, such that
$$\lim_{n\to+\infty} F_{\xi}(\rho_{1,n},\rho_{2,n})=0\,.$$
This implies $\widehat{\t V_{\alpha_1}}=\widehat{\t V_{\alpha_2}}$ and so ${\t V_{\alpha_{1}}}={\t V_{\alpha_{2}}}$. Then, by the definition of $\t V_{\alpha_{m}}$, one obtains $V_{1}-V_{2}=\gamma_{0,2}^{*}\alpha_{2}-\gamma_{0,1}^{*}\alpha_{1}$, where $\gamma_{0,1}$ and $\gamma_{0,2}$ denote the trace operators on   $\Gamma_{1}$ and $\Gamma_{2}$ respectively.
This entails $V_{1}=V_{2}$, $\supp(\alpha_{1})=\supp(\alpha_{2})$ and $\alpha_{1}=\alpha_{2}$.
\end{proof}
\end{section}
\begin{section}{Uniqueness  in inverse acoustic scattering.}
The next lemma probably contains well-known results but we found no proof in the literature.
\begin{lemma}\label{rho} Let $\varrho\ge 0$ satisfy the hypotheses
\be\label{HA1}
\varrho\in L^{\infty}(\RE^{3})\,,\quad\frac1{\varrho}\in L^{\infty}(\RE^{3})\,,\quad|\nabla\varrho\,|\in L^{2}(\RE^{3})\,.\ee
Then
\be\label{sqrt1}
\nabla\frac1{\sqrt\varrho}\in L^{2}(\RE^{3})\quad\text{and}\quad\nabla\frac1{\sqrt{\varrho}}=-\frac12\frac{\nabla\varrho}{\varrho^{3/2}}\,.
\ee
Let us further suppose that $\varrho$ is constant outside some bounded ball $B_{R_{\circ}}$ and there exists an open and bounded set $\Omega_{\varrho}\equiv\Omega\subset B_{R_{\circ}}$ with Lipschitz 
boundary $\Gamma_{\!\varrho}\equiv\Gamma$ such that
\be\label{HA2}|\nabla_{\Omega_{\-/\+}}\varrho\,|\in L^{4}(\Omega_{\-/\+})
\,,\quad \Delta_{\Omega_{\-/\+}}\varrho\in L^{2}(\Omega_{\-/\+})
\ee
Then
\be\label{sqrt2}
\Delta_{\Omega_{\-/\+}}\frac1{\sqrt\varrho}\in L^{2}(\Omega_{\-/\+})
\ee
and
\be\label{gammarho}
%\chi\varrho^{-1/2}|\Omega_{{\-}/{\+}}\in H^{1}_{\Delta}(\Omega_{{\-}/{\+}})\quad\text{and}\quad
[\hat\gamma_{1}](\varrho^{-1/2})=-\frac12\,\frac{[\hat\gamma_{1}]\varrho}{(\gamma_{0}\varrho)^{3/2}}\,,
\ee
where $\gamma_{0}\varrho$ and $[\hat\gamma_{1}]\varrho$ denote the trace on   $\Gamma$ and the jump of the normal derivative across $\Gamma$ respectively.
\end{lemma}
\begin{proof} At first we define the sequence  $\varrho_{n}:=e^{\Delta/n}\varrho$, $n\ge 1$. Then, since the heat semigroup is positivity-preserving, strongly continuous in both $L^{\infty}(\RE^{3})$ and $L^{2}(\RE^{3})$, commutes with $\nabla$ and $e^{t\Delta}(L^{\infty}(\RE^{3}))\subset C^{\infty}(\RE^{3})$ whenever $t>0$, one gets $\varrho_{n}\ge 0$, $\varrho_{n}\in C^{\infty}(\RE^{3})\cap L^{\infty}(\RE^{3})$,  $\nabla\varrho_{n}\in L^{2}(\RE^{3};\RE^{3})$, $\varrho_{n}\to \varrho$ in $L^{\infty}(\RE^{3})$ and $\nabla\varrho_{n}\to \nabla\varrho$ in $L^{2}(\RE^{3};\RE^{3})$ as $n\to+\infty$. Since $\varrho(x)\ge \|1/\varrho\|^{-1}_{L^{\infty}(\RE^{3})}$ for a.e. $x\in\RE^{3}$, $\varrho_{n}(x)$ is definitively stricitly positive uniformly in $x\in\RE^{3}$ and so $\varrho_{n}^{-3/2}\to \varrho^{-3/2}$ in $L^{\infty}(\RE^{3})$. Thus, by $\nabla\sqrt{\varrho_{n}}=-2^{-1}\varrho_{n}^{-3/2}\,\nabla\varrho_{n}\to -2^{-1}\varrho^{-3/2}\,\nabla\varrho$ in $L^{2}(\RE^{3})$, \eqref{sqrt1} follows.\par
By Lemma \ref{1/u}, \eqref{sqrt1} and Remark \ref{rem-leibniz} one gets
\begin{align}\label{lapl}
\Delta\frac1{\sqrt\varrho}=-\frac12\,\nabla\cdot\frac{\nabla\varrho}{\varrho^{3/2}}=
-\frac12\,\frac{\Delta\varrho}{\varrho^{3/2}}-\frac12\,\nabla\varrho\cdot\left(\frac1\varrho\nabla\frac1{\sqrt\varrho}+\frac1{\sqrt\varrho}\nabla\frac1{\varrho}\right)
=-\frac12\,\frac{\Delta\varrho}{\varrho^{3/2}}+\frac34\,\frac{|\nabla\varrho|^{2}}{\varrho^{5/2}}\,.
\end{align}
This, by $\varrho^{-1}\in L^{\infty}(\RE^{3})$ and $\eqref{HA2}$, gives \eqref{sqrt2}.
\par
By Lemma \ref{rho}, one immediately gets $\chi\varrho^{-1/2}|\Omega_{{\-}/{\+}}\in H^{1}_{\Delta}(\Omega_{{\-}/{\+}})$, $\chi\in C^{\infty}_{comp}(\RE^{3})$. Then, by the half Green's formula (see \cite[Theorem 4.4]{McLe}) one has
$$
\langle-\Delta\varrho^{-1/2},v\rangle_{L^{2}(\Omega_{\-})\oplus L^{2}(\Omega_{\+})}=
\langle\nabla\varrho^{-1/2},\nabla v\rangle_{L^{2}(\RE^{3}))}+\langle[\hat\gamma_{1}]\varrho^{-1/2},\gamma_{0}v\rangle_{H^{-1/2}(\Gamma),H^{1/2}(\Gamma)}
$$
for any $v\in H^{1}(\RE^{3})$. Since both $\varrho^{-1}$ and $\varrho^{-1/2}$ belong to $H^{1}_{loc}(\RE^{3})\cap L^{\infty}(\RE^{3})$, one has $\varrho^{-3/2}\in H^{1}_{loc}(\RE^{3})$ and so we can use the above Green's formula in the case $v=\varrho^{-3/2}w$, $w\in C^{\infty}_{comp}(\RE^{3})$. Then, by \eqref{lapl} and \eqref{sqrt1}, one obtains
\begin{align*}
&\langle[\hat\gamma_{1}]\varrho^{-1/2},\gamma_{0}w\rangle_{H^{-1/2}(\Gamma),H^{1/2}(\Gamma)}=
\langle-\Delta\varrho^{-1/2},w\rangle_{L^{2}(\Omega_{\-})\oplus L^{2}(\Omega_{\+})}-
\langle\nabla\varrho^{-1/2},\nabla w\rangle_{L^{2}(\RE^{3}))}\\
=&-\frac12\left(\langle-\Delta\varrho,\varrho^{-3/2}w\rangle_{L^{2}(\Omega_{\-})\oplus L^{2}(\Omega_{\+})}+\frac32\langle\nabla\varrho,\varrho^{-5/2}w\nabla\varrho\rangle_{L^{2}(\RE^{3}))}-
\langle\nabla\varrho,\varrho^{-3/2}\nabla w\rangle_{L^{2}(\RE^{3}))}\right)\\
=&-\frac12\left(\langle-\Delta\varrho,\varrho^{-3/2}w\rangle_{L^{2}(\Omega_{\-})\oplus L^{2}(\Omega_{\+})}-\langle\nabla\varrho,\nabla(\varrho^{-3/2}w)\rangle_{L^{2}(\RE^{3}))}\right)\\
=&\langle[\hat\gamma_{1}]\varrho,\gamma_{0}(\varrho^{-3/2}w)\rangle_{H^{-1/2}(\Gamma),H^{1/2}(\Gamma)}=
\langle[\hat\gamma_{1}]\varrho,\gamma_{0}(\varrho^{-3/2})\gamma_{0}w\rangle_{H^{-1/2}(\Gamma),H^{1/2}(\Gamma)}\,.
\end{align*}
Therefore $[\hat\gamma_{1}](\varrho^{-1/2})=-2^{-1}\gamma_{0}(\varrho^{-3/2})[\hat\gamma_{1}]\varrho$. By Lemma \ref{trace-prod}, $\gamma_{0}(\varrho^{-3/2})=(\gamma_{0}\varrho)^{-3/2}$ and the proof is concluded.
\end{proof}
Now we introduce a further hypothesis on $\varrho$ :
\be\label{HA3}
\frac{[\hat\gamma_{1}]\varrho}{\gamma_{0}\varrho}\in M(H^{s}(\Gamma),H^{-s}(\Gamma))\,,\quad s\in (0,1/2)\,.
\ee
In particular, by Remark \ref{multi}, hypothesis \eqref{HA3} holds true whenever
$$
\frac{[\hat\gamma_{1}]\varrho}{\gamma_{0}\varrho}\in L^{p}(\Gamma) \quad\text{ for some $p>2$.}
$$
%If $\varrho\ge c>0$ is continuous, so that $\frac1{\gamma_{0}\varrho}\in L^{\infty}(\Gamma)$, then
%hypothesis \eqref{HA3} holds true whenever $${[\hat\gamma_{1}]\varrho}\in L^{p}(\Gamma)\quad \text{ for some $p>2$.}$$
\begin{corollary}\label{corollary} If $\varrho\ge 0$ satisfies\eqref{HA1}, \eqref{HA2} and \eqref{HA3}, then $\varphi:=\varrho^{-1/2}$ satisfies \eqref{H1} and \eqref{H2}.
\end{corollary}
\begin{proof}
Hypotheses \eqref{H1} are consequence of Lemma \ref{rho}. By Lemma \ref{trace-prod}, one has $\gamma_{0}(\varrho^{-1/2})=(\gamma_{0}\varrho)^{-1/2}$ and then hypothesis \eqref{H2} follows from \eqref{gammarho}.
\end{proof}
Let us now take $v\ge 0$ such that $v^{-1}\in L^{\infty}(\RE^{3})$, suppose that $\varrho(x)=v(x)=1$ whenever $x$ lies outside some large ball $B_{R_{\circ}}$ and set
$$
\varphi:=\frac1{\sqrt\varrho}\,,\qquad V_{\varphi}:= \frac1\varphi\,\left(\Delta_{\Omega_{\-}}(\varphi|\Omega_{\-})+\Delta_{\Omega_{\+}}(\varphi|\Omega_{\+})\right)\,,
$$
$$
V_{v,\omega}:=\omega^{2}\left(1-\frac1{v^{2}}\right)\,,\qquad V_{\varphi,v,\omega}:=V_{\varphi}+V_{v,\omega}\,,
$$
so that $V_{\varphi}\in L^{2}_{comp}(\RE^{3})$ and $V_{v,\omega}\in L^{\infty}_{comp}(\RE^{3})$.
By Lemma \ref{bounded} and Theorem \ref{teo-acous}, there is a well defined correspondence between the outgoing   eigenfunctions $\psi^{\varphi,v,\omega}_{\omega} $ of $A_{\varphi}+V_{v,\omega}\equiv A_{V_{\varphi,v,\omega},\alpha_{\varphi}}$ provided in Theorem \ref{gen} and the acoustic eigenfunctions $u^{\varrho,v}_{\omega}$ such that
\be\label{ac}
\omega^{2}u^{\varrho,v}_{\omega}+v^{2}\varrho\nabla\!\cdot\!\left(\frac1\varrho\,\nabla u^{\varrho,v}_{\omega}\right)=0\,.
\ee
Such a relation is given by $u^{\varrho,v}_{\omega}=\varrho\, \psi^{{\varphi,v,\omega}}_{\omega}$. Notice that $u(t,x):=e^{-i\omega t}u^{\varrho,v}_{\omega}(x)$ is a fixed-frequency solution of the acoustic wave equation
$$
\partial_{tt}u=v^{2}\varrho\nabla\!\cdot\!\left(\frac1\varrho\,\nabla u\right)\,.
$$
By Corollary \ref{ff}, since $u^{\varrho,v}_{\omega}=\psi^{{\varphi,v,\omega}}_{\omega}$ outside $B_{R_{\circ}}$, the eigenfunction $u^{\varrho,v}_{\omega}$ has the asymptotic behavior
$$
u^{\varrho,v}_{\omega}(x)=e^{-i\omega\hat\xi\cdot x}+\frac{e^{i\omega|x|}}{|x|}\, u^{\infty}_{\varrho,v}(\omega,\hat\xi,\hat x)+O(|x|^{-2})\,, \quad |x|\gg R_{\circ}\,,
$$
uniformly in all directions $\hat x:=\frac{x}{|x|}$, where the far-field pattern $u^{\infty}_{\varrho,v}$
is related to the scattering amplitude for the Schr\"odinger operator $A_{V_{\varphi,v,\omega},\alpha_{\varphi}}$ by
\be\label{ff-sa}
u^{\infty}_{\varrho,v}(\omega,\hat\xi,\hat\xi')=\frac1{(2\pi)^{3/2}}\, s_{V_{\varphi,v,\omega}}(\omega,\hat\xi,\hat\xi')\,.
\ee
By Theorem \ref{gen}, $u^{\varrho,v}_{\omega}$ is the unique solution of the stationary acoustic equation \eqref{ac} such that the scattered field $u^{\varrho,v}_{\omega,sc}(x):=u^{\varrho,v}_{\omega}(x)-e^{-i\omega\hat\xi\cdot x}$ satisfies the outgoing   Sommerfeld radiation condition.
\begin{remark}\label{SLA} An more explicit characterization of a class of function $\varphi$ satisfying hypotheses \eqref{HA1}, \eqref{HA2} and \eqref{HA3} is the following:
$$
\varrho(x)=1+\chi_{\circ}(\varrho_{\circ}+\SL\xi)\,,\qquad \SL\xi(x):=\int_{\Gamma}\frac{\xi(y)\,d\sigma_{\Gamma}(y)}{4\pi\,|x-y|}
$$
where $\chi_{\circ}\in C_{comp}^{\infty}(\RE^{3})$, $\chi_{\circ}=1$ on some large ball containing $\Omega$, $\varrho_{\circ}\in H^{2}(\RE^{3})$ and $\xi\in L^{p}(\Gamma)$, $p>8/3$. This includes the case where the normal derivative of $\varrho$ has a jump across $\Gamma$ which is locally supported on a closed subset $\Sigma\subset\Gamma$. By the same reasoning as in Remark \ref{SL} one has $|\nabla\varrho|\in L^{2}(\RE^{3})$, $\Delta_{\Omega_{\-/\+}}\varrho\in L^{2}(\Omega_{\-/\+})$ and $[\hat\gamma_{1}]\varrho/\gamma_{0}\varrho\in L^{p}(\Gamma)\subset M(H^{s}(\Gamma),H^{-s}(\Gamma))$. Moreover, by $H^{2}(\RE^{3})\subset W^{1,6}(\RE^{3})$, by
$\SL\xi\in W^{1+1/p-\epsilon,p}(\Omega_{\-/\+})$ (see \cite[Theorem 3.1]{FMM}) and by
$W^{1+1/p-\epsilon,p}(\Omega_{\-/\+})\subset W^{1,4}(\Omega_{\-/\+})$ whenever $p>8/3$,
one has  $|\nabla_{\Omega_{\/\+}}\varrho|\in L^{4}(\Omega_{\-/\+})$.
\end{remark}
Thanks to Theorem \ref{uniq}, one gets the following uniqueness result in acoustic scattering:
\begin{theorem}\label{Theorem-Acoustic} Let $\varrho_{1}\ge 0$, $\varrho_{2}\ge 0$ satisfy hypotheses \eqref{HA1}, \eqref{HA2}, \eqref{HA3} (see for example Remark \ref{SLA}) and let $v_{1}\ge 0$, $v_{2}\ge 0$ such that  $v_{1}^{-1},v_{2}^{-1}\in L^{\infty}(\RE^{3})$. Suppose $\RE^{3}\backslash\overline\Omega_{\varrho_{1}}$ and $\RE^{3}\backslash\overline\Omega_{\varrho_{2}}$ are connected and that
$\varrho_{1}(x)=\varrho_{2}(x)=v_{1}(x)=v_{1}(x)=1$ whenever $x$ lies outside some large ball $B_{R_{\circ}}\supset(\Omega_{\varrho_{1}}\cup\Omega_{\varrho_{2}})$. Then, given $\omega\not=\t\omega\not=0$,
$$
\begin{cases} u^{\infty}_{\varrho_{1},v_{1}}(\omega,\cdot,\cdot)=u^{\infty}_{\varrho_{2},v_{2}}(\omega,\cdot,\cdot)\\ u^{\infty}_{\varrho_{1},v_{1}}(\t\omega,\cdot,\cdot)=u^{\infty}_{\varrho_{2},v_{2}}(\t\omega,\cdot,\cdot)&\end{cases} \quad\Longrightarrow\qquad \begin{cases}\varrho_{1}=\varrho_{2}&\\v_{1}=v_{2}\,.&\end{cases}
$$
\end{theorem}
\begin{proof} Let $\varphi_{m}=\varrho_{m}^{-1/2}$, $m=1,2$. By \eqref{ff-sa} and Theorem \ref{uniq}, one gets
\be\label{J}
\alpha_{\varphi_{1}}=\frac{[\hat\gamma_{1}]\varphi_{1}}{\gamma_{0}\varphi_{1}}=\frac{[\hat\gamma_{1}]\varphi_{2}}{\gamma_{0}\varphi_{2}}=\alpha_{\varphi_{2}}
\ee
Furthermore, considering any open and bounded $\Omega$ with Lipschitz boundary $\Gamma$ such that $\supp(\alpha_{\varphi_{1}})=\supp(\alpha_{\varphi_{2}})\subseteq \Gamma$ one has, in both $L^{2}(\Omega_{\-})$ and $L^{2}(\Omega_{\+})$,
$$
\frac{\Delta\varphi_{1}}{\varphi_{1}}+\omega^{2}\left(1-\frac1{v_{1}^{2}}\right)=
\frac{\Delta\varphi_{2}}{\varphi_{2}}+\omega^{2}\left(1-\frac1{v_{2}^{2}}\right),
$$
$$
\frac{\Delta\varphi_{1}}{\varphi_{1}}+\t\omega^{2}\left(1-\frac1{v_{1}^{2}}\right)=
\frac{\Delta\varphi_{2}}{\varphi_{2}}+\t\omega^{2}\left(1-\frac1{v_{2}^{2}}\right).
$$
Thus, since $\omega\not=\t\omega\not=0$, one has $v_{1}=v_{2}$ and
$$
{(\varphi_{2}|{\Omega_{\-/\+}})}\Delta_{\Omega_{\-/\+}}(\varphi_{1}|{\Omega_{\-/\+}})={(\varphi_{1}|{\Omega_{\-/\+}})}\Delta_{\Omega_{\-/\+}}{(\varphi_{2}}|{\Omega_{\-/\+}})\,.
$$
Let us set $u_{\-/\+}:=(\varphi_{2}-\varphi_{1})|\Omega_{\-/\+}$, so that
$$
(-\Delta_{\Omega_{\-/\+}}+V_{\-/\+})u_{\-/\+}=0\,,\quad V_{\-/\+}:=\frac{\Delta_{\Omega_{\-/\+}}(\varphi_{1}|\Omega_{\-/\+})}{{\varphi_{1}|\Omega_{\-/\+}}}\,.
$$
Since $u_{\+}=0$ outside $B_{R_{\circ}}$ and $V_{\+}\in L^{2}_{comp}(\Omega_{\+})$, the unique continuation principle (see e.g. \cite{JerKo}) leads us to $\varphi_{1}=\varphi_{2}$ on $\Omega_{\+}$. In addition, due to Corollary \ref{corollary}, both $\varphi_{1}$ and $\varphi_{2}$ belong to $H^{1}(\RE^{3})$; then the previous identity yields $\gamma_{0}\varphi_{1}=\gamma_{0}\varphi_{2}$ and $\hat\gamma^{\+}_{1}\varphi_{1}=\hat\gamma^{\+}_{1}\varphi_{2}$. Hence, setting  $u:=u_{\-}\oplus u_{\+}$, one has $[\gamma_{0}]u=0$ and, by \eqref{J}, $[\hat\gamma_{1}]u=0$. By elliptic regularity, $u_{\-}\in H^{2}(\Omega_{\-})$;  so $u$ belongs to $H^{2}(\RE^{3})$ and solves $(-\Delta+V_{\-}\oplus V_{\+})u=0$. Using again the unique continuation principle, this entails $u=0$. Therefore $\varphi_{1}=\varphi_{2}$, i.e. $\varrho_{1}=\varrho_{2}$.
\end{proof}
\end{section}

\section*{References}
%\phantomsection
\addcontentsline{toc}{section}{References}%
\markboth{References}{References}%

\end{document}